\documentclass[a4paper,11pt]{article}

\usepackage{graphicx}
\usepackage{amsmath}
\usepackage{amsfonts}
\usepackage[compact]{titlesec}
\usepackage{amsthm}
\usepackage{amssymb}
\usepackage{color}
\usepackage{mathtools}
\usepackage{geometry}
\usepackage{times}
\usepackage{algorithm2e}

\newenvironment{algo}{%
  \algorithm
}{%
  \endalgorithm
}

\usepackage{verbatim}

\newcommand{\R}{\mathbb{R}}
\newcommand{\T}{\mathcal{T}}

\newcommand{\Epsilon}{\mathcal{E}}

\newcommand{\Vms}{V^{\operatorname*{ms}}}

\newtheorem{theorem}{Theorem}[section]

\newtheorem{lemma}[theorem]{Lemma}

\theoremstyle{definition}
\newtheorem{definition}[theorem]{Definition}
\newtheorem{remark}[theorem]{Remark}
\newtheorem{conclusion}[theorem]{Conclusion} %
\newcommand{\bfs}[1]{{\boldsymbol #1}}

\definecolor{light-gray}{gray}{0.69}
\definecolor{light-red}{rgb}{1.0,0.4,0.4}

\definecolor{light-blue}{rgb}{0.4,0.45,1}
\definecolor{light-green}{rgb}{0.5,0.8,0.0}
\definecolor{dark-green}{rgb}{0.0,0.4,0.0}
\definecolor{dark-red}{rgb}{1.0,0.3,0.3}

\definecolor{dark-gray}{gray}{0.59}
\definecolor{very-dark-gray}{gray}{0.39}
\definecolor{lighter-red}{rgb}{1.0,0.6,0.6}

\definecolor{ocker_hell}{rgb}{0.75,0.7,0.4}
\definecolor{gelb_dunkel}{rgb}{0.75,0.7,0.0}
\definecolor{gruen_hell}{rgb}{0.5,0.8,0.0}

\definecolor{dark-blue}{rgb}{0.0,0.0,0.5}
\definecolor{new-blue}{rgb}{0.0,0.0,0.8}

\definecolor{lila}{rgb}{0.5,0.0,0.5}
\definecolor{dark-red}{rgb}{0.5,0.0,0.0}

\makeatletter
\renewcommand*{\@seccntformat}[1]{\csname the#1\endcsname\hspace{0.3em}}

\renewcommand\section{\@startsection{section}{1}{0pt}%
                                          {8pt plus 4pt minus 4pt}%
                                          {.01pt}%
                                          {\bf}}

\renewcommand\subsection{\@startsection{subsection}{2}{0pt}%
                                          {-1.3ex plus -0.5ex minus -.9ex}%
                                          {-2pt}%
                                          {\bf} }

\renewcommand\subsubsection{\@startsection{subsubsection}{3}{0pt}%
                                          {-.5ex plus -.5ex minus -.2ex}%
                                          {-10pt}%
                                          {\bf\em}}
\makeatother

\titlespacing{\section}{0pt}{2ex}{1ex}
\titlespacing{\subsection}{0pt}{-1ex}{0ex}
\titlespacing{\subsubsection}{0pt}{0.5ex}{0ex}

\begin{document}

\renewcommand{\thefootnote}{\fnsymbol{footnote}}\setcounter{footnote}{0}

\begin{center}
{\LARGE On Multiscale Methods in Petrov-Galerkin formulation
\footnote{D. Elfverson and P. Henning were supported by The G\"{o}ran Gustafsson Foundation and The Swedish Research Council.}}\\[2em]
\end{center}

\renewcommand{\thefootnote}{\fnsymbol{footnote}}
\renewcommand{\thefootnote}{\arabic{footnote}}

\begin{center}
{\large Daniel Elfverson\footnote[1]{Department of Information Technology, Uppsala University, Box 337, SE-751 05 Uppsala, Sweden},
           Victor Ginting\footnote[2]{Department of Mathematics, University of Wyoming, Laramie, Wyoming 82071, USA}, 
           Patrick Henning\footnote[3]{Section de Math\'{e}matiques, \'{E}cole polytechnique f\'{e}d\'{e}rale de Lausanne, 1015 Lausanne, Switzerland}}\\[2em]
\end{center}

\begin{center}
{\large{\today}}
\end{center}

\begin{center}
\end{center}

\begin{abstract}
In this work we investigate the advantages of multiscale methods in Petrov-Galerkin (PG) formulation in a general framework. The framework is based on a localized orthogonal decomposition of a high dimensional solution space into a low dimensional multiscale space with good approximation properties and a high dimensional remainder space{, which only contains negligible fine scale information}. The multiscale space can then be used to obtain accurate Galerkin approximations. As a model problem we consider the Poisson equation. We prove that a Petrov-Galerkin formulation does not suffer from a significant loss of accuracy, and still preserve the convergence order of the original multiscale method. We also prove inf-sup stability of a PG Continuous and a Discontinuous Galerkin Finite Element multiscale method. Furthermore, we demonstrate that the Petrov-Galerkin method can decrease the computational complexity significantly, allowing for more efficient solution algorithms. As another application of the framework, we show how the Petrov-Galerkin framework can be used to construct a locally mass conservative solver for
two-phase flow simulation that employs
 the Buckley-Leverett equation. To achieve this, we couple a PG Discontinuous Galerkin Finite Element method with an upwind scheme for a hyperbolic conservation law.
\end{abstract}



\section{Introduction}
\setcounter{equation}{0}
In this contribution we consider linear elliptic problems with a heterogenous and highly variable diffusion coefficient $A$ as arisen often in hydrology or in material sciences. In the following, we are looking for $u$ which solves 
\begin{equation*}\label{eq:model}
  \begin{aligned}
    -\nabla \cdot A\nabla u &= f\quad \hspace{6pt}\text{in }\Omega, \\
    u & = 0 \quad \hspace{7pt}\text{on }\partial \Omega,
  \end{aligned}
\end{equation*}
in a weak sense. Here, we denote
\begin{itemize}
\item[(A1)] $\Omega\subset\mathbb{R}^{d}$, $d=1,2,3$, a bounded Lipschitz domain with a piecewise polygonal boundary,
\item[(A2)] $f \in L^2(\Omega)$ a source term, and
\item[(A3)] $A\in L^{\infty}(\Omega,\mathbb{R}^{d\times d}_{sym})$ a symmetric matrix-valued function with uniform spectral bounds $\beta_0\geq\alpha_0>0$, 
\begin{equation}\label{e:spectralbound}
\sigma(A(x))\subset [\alpha_0,\beta_0]\quad\text{for almost all }x\in \Omega.
\end{equation}
We call the ratio $\beta_0/\alpha_0$ the {\it contrast} of $A$.
\end{itemize}
Under assumptions (A1)-(A3) and by the Lax-Milgram theorem, there exists a unique weak solution $u \in H^{1}_0(\Omega)$ to
\begin{align}
\label{exact-solution} a(u,v) = (f,v) \qquad \mbox{for all } v \in H^1_0(\Omega),
\end{align}
where
\begin{align*}
a(v,w):= \int_{\Omega} A \nabla v \cdot \nabla w \quad \mbox{and} \quad (v,w):=(v,w)_{L^2(\Omega)}.
\end{align*}

The problematic term in the equation is the diffusion matrix $A$, which is known to exhibit very fast variations on a very fine scale (i.e. it has a multiscale character). These variations can be highly heterogenous and unstructured, which is why it is often necessary to resolve them globally by an underlying computational grid that matches the said heterogeneity. Using standard finite element methods, this results in high dimensional solution spaces and hence an enormous computational demand, which often cannot be handled even by today's computing technology. Consequently, there is a need for alternative methods, so called multiscale methods, which can either operate below linear computational complexity by using local representative elements
(cf. 
\cite{Abdulle:2005,Abdulle:E:Engquist:Vanden-Eijnden:2012,E:Engquist:2003,Gloria:2006,Gloria:2011,Henning:Ohlberger:2009,Ohlberger:2005}) or which can split the original problem into very localized subproblems that cover $\Omega$ but that can be solved cheaply and independent from each other (cf. \cite{MR2801210,bgp-cam-13,MR2281625,Efendiev:Hou:Ginting:2004,MR2062582,Hou:Wu:1997,Owhadi:Zhang:Berlyand:2013,Hughes:1995,MR1660141,Larson:Malqvist:2005,Malqvist:2011,Owhadi:Zhang:2011}).

In this paper, we focus on a rather recent approach called Localized Orthogonal Decomposition (LOD) that was introduced by M{\aa}lqvist and Peterseim \cite{Malqvist:Peterseim:2011} and further generalized in \cite{HP12,Henning:Malqvist:2013}. 

We consider a coarse space $V_H$, which is low-dimensional but possibly inadequate for finding a reliable Galerkin approximation to the multiscale solution of problem (\ref{exact-solution}). The idea of the method is to start from this coarse space
and to update the corresponding set of basis functions step-by-step to improve the approximation properties of the space. In a summarized form, this can be described in four steps: 1) define a (quasi) interpolation operator $I_H$ from $H^1_0(\Omega)$ onto $V_H$, 2) information in the kernel of the interpolation operator is considered to be negligible (having a small $L^2$-norm), 3) hence define the space of negligible information by the kernel of this interpolation, i.e. $W:=$kern$(I_H)$, and 4) find the orthogonal complement of $W$ with respect to a scalar product $a_h(\cdot,\cdot)$, where $a_h(\cdot,\cdot)$ describes a discretization of the problem to solve. In many cases, it can be shown, that this (low dimensional) orthogonal complement space has very accurate approximation properties with respect to the exact solution. Typically, the computation of the orthogonal decomposition is localized to small patches in order to reduce the computational complexity.

So far, the concept of the LOD has been successfully applied to nonlinear elliptic problems \cite{Henning:Malqvist:Peterseim:2013},
eigenvalue problems \cite{Malqvist:Peterseim:2012} and
the nonlinear Schr\"odinger equation \cite{Henning:Malqvist:Peterseim:2013-2}. Furthermore, it was combined with a discontinuous Galerkin method \cite{EGM13,EGMP13} and extended to the setting of partition of unity methods \cite{Henning:Morgenstern:Peterseim:2013}.

In this work, we are concerned with analyzing the LOD framework in Petrov-Galerkin formulation, i.e. for the case that the discrete trial and test spaces are not identical. We show that an LOD method in Petrov-Galerkin formulations still preserves the convergence rates of the original formulation of the method. At the same time, the new method can exhibit significant advantages, such as decreased computational complexity and mass conservation properties. In this paper, we discuss these advantages in detail; we give examples for realizations and present numerical experiments. In particular, we apply the proposed framework to design a locally conservative multiscale solver for the simulation of two-phase flow models as governed by the Buckley-Leverett equation. We remark that employing  Petrov-Galerkin variational frameworks in the construction and analysis of multiscale methods for solving elliptic problems in heterogeneous media has been investigated in the past, see for example \cite{Hou:Wu:Zhang:2004} and \cite{MR2062582}.

The rest of the paper is organized as follows. Section~\ref{sec:disc} lays out  the setting and notation for the formulation of the multiscale methods that includes the description of two-grid discretization and the Localized Orthogonal Decomposition (LOD). In Section~\ref{sec:methods}, we present the multiscale methods based on the LOD framework, starting from the usual Galerkin variational equation and concentrating further on the Petrov-Galerkin variational equation that is the main contribution of the paper. We establish in this section that the Petrov-Galerkin LOD (PG-LOD) exhibits the same convergence behavior as the usual Galerkin LOD (G-LOD). Furthermore, we draw a contrast in the aspect of practical implementation that makes up a strong advantage of PG-LOD in relative comparison to G-LOD. The other advantage of the PG-LOD which cannot be achieved with G-LOD is the ability to produce a locally conservative flux field at the elemental level when discontinuous finite element is utilized. We also discuss in this section an application of the PG-LOD for solving the pressure equation in the simulation of two-phase flow models to demonstrate this particular advantage. Section~\ref{section:numerical:experiments} gives two sets of numerical experiment: one that confirms the theoretical finding and the other demonstrating the application of PG-LOD in the two-phase flow simulation. We present the proofs of the theoretical findings in Section~\ref{section:proofs:pg:lod}.

\section{Discretization}
\label{sec:disc}
\setcounter{equation}{0}
In this section we introduce notations that are required for the formulation of the multiscale methods.

\subsection{Abstract two-grid discretization}
\label{subsection:two-grid:disc}

We define two different meshes on $\Omega$. The first mesh is a 'coarse mesh' and is denoted by $\T_H$,
where $H>0$ denote the maximum diameter of all elements of $\T_H$.  The second mesh is a 'fine mesh' denoted by $\T_h$ with
$h$ representing the maximum diameter of all elements of $\T_h$. By 'fine' we mean that any variation of the coefficient $A$ is resolved within this grid, leading to a high dimensional discrete space that is associated with this mesh. The mesh $\T_h$ is assumed to be a (possibly non-uniform) refinement of $\T_H$. Furthermore, both grids are shape-regular and conforming partitions of $\Omega$ and we assume that $h<H/2$. For the subsequent methods to make sense, we also assume that each element of $\T_H$ is at least twice uniformly refined to create $\T_h$. The set of all Lagrange points (vertices) of $\T_\star$ is denoted by $\mathcal{N}_\star$, and the set of interior Lagrange points is denoted by $\mathcal{N}_\star^0$, where $\star$ is either
$H$ or $h$.

Now we consider an abstract discretization of the exact problem (\ref{exact-solution}). For this purpose, we let $V_h$ denote a high dimensional discrete space in which we seek an approximation $u_h$ of $u$. A simple example would be the classical $P1$ Lagrange Finite Element space associated with $\T_h$. However, note that we do not assume that $V_h$ is a subspace of $H^1_0(\Omega)$. In fact, later we give an example for which $V_h$ consists of non-continuous piecewise linear functions. Next, we assume that we are interested in solving a fine scale problem, that can be characterized by a scalar product $a_h(\cdot,\cdot)$ on $V_h$. Accordingly, a method on the coarse scale can be described by some $a_H(\cdot,\cdot)$, which we specify by assuming
\begin{itemize}
\item[(A4)] $a_\star(\cdot,\cdot)$ is a scalar product on $V_\star$ where $\star$ is either $h$ or $H$.
\end{itemize}
This allows us to define the abstract reference problem stated below.

\begin{definition}[Fine scale reference problem] 
We call $u_h \in V_h$ the fine scale reference solution if it solves
\begin{align}
\label{fine-grid}a_h(u_h,v_h) = (f, v_h)_{L^2(\Omega)} \qquad \mbox{for all } v_h \in V_h,
\end{align}
where $a_h(\cdot,\cdot)$ 'describes the method'. It is implicitly assumed that problem (\ref{fine-grid}) is of tremendous computational complexity and cannot be solved by available computing resources i
n a convenient time. 
\end{definition}

A simple example of $a_h(\cdot,\cdot)$ is $a_h(v_h,w_h)= a_H(v_h,w_h)=a(v_h, w_h)$. A more complex example is the $a_h(\cdot,\cdot)$ that stems from a discontinuous Galerkin approximation, in which case $a_h(\cdot,\cdot)$ is different from $a_H(\cdot,\cdot)$. 
The goal is to approximate problem (\ref{fine-grid}) by a new problem that reaches a comparable accuracy but one that can be solved with a significantly lower computational demand.

\subsection{Localized Orthogonal Decomposition}
\label{subsection:lod-framework}
In this subsection, we introduce the notation that is required in the formulation of the multiscale method. In particular, we introduce an orthogonal decomposition of the high dimensional solution space $V_h$ into the orthogonal direct sum of a low dimensional space with good approximation properties and a high dimensional remainder space.
For this purpose, we make the following abstract assumptions.
\begin{itemize}
\item[(A5)] $||| \cdot |||_h$ denotes a norm on $V_h$ that is equivalent to the norm that is induced by $a_h(\cdot,\cdot)$, hence there exist generic constants $0<\alpha \le \beta$ such that
\begin{align*}
\alpha ||| v_h |||_h^2 \le a_h(v_h,v_h) \quad \mbox{and} \quad a_h(v_h,w_h) \le \beta||| v_h |||_h ||| w_h |||_h \quad \mbox{for all } v_h,w_h \in V_h.
\end{align*}
In the same way, $||| \cdot |||_H$ denotes a norm on $V_H$ (equivalent to the norm induced by $a_H(\cdot,\cdot)$). Furthermore, we let $C_{H,h}$ denote the constant with $|||v|||_{H} \le C_{H,h} |||v|||_{h}$ for all $v\in V_h$. Note that $C_{H,h}$ might degenerate for $h\rightarrow 0$.
\item[(A6)] The coarse space $V_H \subset V_h$ is a low dimensional subspace of $V_h$ that is associated with $\T_H$.
\item[(A7)] Let $I_H : V_h \rightarrow V_H$ be an $L^2$-stable quasi-interpolation (or projection) operator with the properties
\begin{itemize}
\item there exists a generic constant $C_{I_H}$ (only depending on the shape regularity of $\T_H$ and $\T_h$) such that for all $v_h \in V_h$ and $v_H \in V_H$
\begin{align*}
\nonumber \|v_h-I_H(v_h)\|_{L^{2}(\Omega)} &\leq C_{I_H} H ||| v_h |||_h, \quad \mbox{and} \quad
\nonumber ||| I_H(v_h) |||_H \leq C_{I_H} ||| v_h |||_h,\\
 \|v_H-I_H(v_H)\|_{L^{2}(\Omega)} &\leq C_{I_H} H ||| v_H |||_H , \quad \mbox{and} \quad
\nonumber \| I_H(v_H) \|_{L^{2}(\Omega)} \leq C_{I_H} ||| v_H |||_H.
\end{align*}
\item the restriction of $I_H$ to $V_H$ is an isomorphism with $||| \cdot |||_H$-stable inverse, i.e. we have $v_H= (I_H \circ ( I_H\vert_{V_H})^{-1})(v_H)$ for $v_H \in V_H$ and the exists a generic $C_{I_H^{-1}}$ such that
\begin{equation*}
\nonumber ||| ( I_H\vert_{V_H})^{-1}(v_H) |||_H \leq C_{I_H^{-1}} ||| v_H |||_H \quad \mbox{for all } v_H \in V_H.
\end{equation*}
\end{itemize}
\end{itemize}
Typically, $L^2$-projections onto $V_H$ can be verified to fulfill assumption (A7). Similarly, $I_H$ can be a quasi-interpolation of the Cl\'ement-type that is related to the $L^2$-projection. An example for this case is given in equation (\ref{def-weighted-clement}) below. Alternatively, $I_H$ can be also constructed from local $L^2$-projections as it is done for the classical Cl\'ement interpolation. Nodal interpolations typically do not satisfy (A7).

Using the assumption that $(I_H)_{|V_H}: V_H \rightarrow V_H$ is an isomorphism (i.e. assumption (A7)), a splitting of the space $V_h$ is given by the direct sum
\begin{align}
\label{def-W_h}V_h = V_H  \oplus W_h, \quad \mbox{with} \enspace W_h := \{ v_h \in V_h | \hspace{2pt} I_H(v_h) = 0 \}.
\end{align}
Observe that the 'remainder space' $W_h$ contains all fine scale features of $V_h$ that cannot be expressed in the coarse space $V_H$. 

Next, consider the $a_h(\cdot,\cdot)$-orthogonal projection $P_h:V_h \rightarrow W_h$ that fulfills:
\begin{align}
\label{projection-orthogonal} a_h( P_h(v_h) , w_h ) = a_h ( v_h , w_h ) \qquad \mbox{for all } w_h \in W_h.
\end{align}
Since $V_h = V_H  \oplus W_h$, we have that $\Vms_{\Omega}:=\mbox{kern}(P_h)=(1-P_h)(V_H)$ induces the $a_h(\cdot,\cdot)$-orthogonal splitting
$$V_h = \Vms_{\Omega} \oplus W_h.$$

Note that $\Vms_{\Omega}$ is a low dimensional space in the sense that it has the same dimension as $V_H$. As shown for several applications (cf. \cite{Malqvist:Peterseim:2012,Henning:Malqvist:Peterseim:2013,Henning:Malqvist:Peterseim:2013-2}) the space $\Vms_{\Omega}$ has very rich approximation properties in the $|||\cdot |||_h$-norm. However, it is very expensive to assemble $\Vms_{\Omega}$, which is why it is practically necessary to localize the space $W_h$ (respectively localize the projection). This is done using admissible patches of the following type.

\begin{definition}[Admissible patch]
\label{admissible-patch}
For any coarse element $T \in \T_H$, we say that the open and connected set $U(T)$ is an {\it admissible patch} of $T$, if $T \subset U(T) \subset \Omega$ and if it consists of elements from the fine grid, i.e.
\begin{align*}
    U(T) = \operatorname{int}\bigcup_{{\tau} \in \T_h^U} \overline{\tau}, \quad \mbox{where} \enspace
    \T_h^U \subset \T_h.
\end{align*}
\end{definition}
\noindent
It is now relevant to define the restriction of $W_h$ to an admissible patch $U(T)\subset \Omega$ by
$$
\mathring{W}_h(U(T)):=\{ v_h \in W_h| \hspace{2pt} v_h=0 \enspace \mbox{in } \Omega \setminus U(T) \}.
$$ 

A general localization strategy for the space $\Vms_{\Omega}$ can be described as follows (see \cite{Henning:Malqvist:2013} for a special case of this localization and \cite{Malqvist:Peterseim:2011} for a different localization strategy).

\begin{definition}[Localization of the solution space]
\label{localization-of-solution-space}
Let the bilinear form $a_h^T( \cdot , \cdot )$ be a localization of $a_h(\cdot,\cdot)$ on $T \in \T_H$ in the sense that
\begin{align}
\label{localization-of-a_h-to-a_h-T}a_h ( v_h , w_h ) &= \sum_{T \in \T_H } a_h^T( v_h, w_h ),
\end{align}
where $a_h^T(\cdot,\cdot)$ acts only on $T$ or a small environment of $T$.
Let furthermore $U(T)$ be an admissible patch associated with $T\in \T_H$. Let $Q_h^T: V_h \rightarrow \mathring{W}_h(U(T))$ be a local correction operator
that is defined as finding $Q_h^{T}(\phi_h) \in \mathring{W}_h(U(T))$  satisfying
\begin{align}
\label{local-corrector-problem} a_h( Q_h^{T}(\phi_h), w_h ) = - a_h^T ( \phi_h , w_h ) \qquad \mbox{for all } w_h \in \mathring{W}_h(U(T)),
\end{align}
where $\phi_h \in V_h$.
The global corrector is given by
\begin{align}
\label{global-corrector}Q_h(\phi_h):=\sum_{T\in \T_H} Q_h^{T}(\phi_h).
\end{align}
A (localized) generalized finite element space is defined as
$$
\Vms:=\{ \Phi_H + Q_h(\Phi_H) | \hspace{2pt} \Phi_H \in V_H \}.
$$ 
\end{definition}
The variational formulation \eqref{local-corrector-problem} is called the corrector problem associated with $T \in \T_H$. Solvability of each of these problems is guaranteed by the Lax-Milgram Theorem. By its nature, the system matrix corresponding to (\ref{local-corrector-problem}) is localized to the patch $U(T)$ since the support of $w_h$ is in $U(T)$. Furthermore, each of (\ref{local-corrector-problem}) pertaining to $T \in \T_H$ is designed to be elementally independent and thus attributing to its immediate parallelizability. The corrector problems are solved in a preprocessing step and can be reused for different source terms and for different realization of the LOD methods. Since $\Vms$ is a low dimensional space with locally supported basis functions, solving a problem in $\Vms$ is rather inexpensive. Normally, the solutions $Q_h^{T}(\phi_h)$ of (\ref{local-corrector-problem}) decays exponentially to zero outside of $T$. This is the reason why we can hope for good approximations even for small patches $U(T)$. Later, we quantify this decay by an abstract assumption (which is known to hold true for many relevant applications).

\begin{remark}\label{remark-on-projection}
If $U(T)=\Omega$ for all $T\in \T_H$, then $Q_h=- P_h$, where $P_h$ is the orthogonal projection given by (\ref{projection-orthogonal}). In this sense, $\Vms$ is localization of the space $\Vms_{\Omega}$. This can be verified using (\ref{localization-of-a_h-to-a_h-T}), which yields
\begin{align*}
a_h( \phi_h + Q_h(\phi_h), w_h ) = \sum_{T \in \T_H } \left( a_h^T( \phi_h, w_h ) + a_h( Q_h^{T}(\phi_h), w_h ) \right) = 0 \qquad \mbox{for all } w_h \in W_h.
\end{align*}
By uniqueness of the projection, we conclude $Q_h=- P_h$.
\end{remark}

The above setting is used to construct the multiscale methods utilizing the Localized Orthogonal Decomposition Method (LOD) as e.g. done in \cite{Henning:Malqvist:2013,Malqvist:Peterseim:2011} for the standard finite element formulation and a corresponding Petrov-Galerkin formulation.

\section{Methods and properties}
\label{sec:methods}
\setcounter{equation}{0}
In this section, we state the LOD in Galerkin and in Petrov-Galerkin formulation along with their respective a priori error estimates and the inf-sup stability. In the last part of this section, we give two explicit examples and discuss the advantages of the Petrov-Galerkin formulation.
Subsequently we use the notation $a \lesssim b$ to abbreviate $a\leq Cb$, where $C$ is a constant that is independent of the mesh sizes $H$ and $h$; and which is independent of the possibly rapid oscillations in $A$.

In order to state proper a priori error estimates, we describe the notion of 'patch size' and how the size of $U(T)$ affects the final approximation.
All the stated theorems on the error estimates of the LOD methods are proved in Section~\ref{section:proofs:pg:lod}.

\begin{definition}[Patch size]
\label{category-k}Let $k\in \mathbb{N}_{>0}$ be fixed. We define patches $U(T)$ that consist of the element $T$ and $k$-layers of coarse element around it. For all $T\in\T_H $, we define element patches in the coarse mesh $\T_H$ by
\begin{equation}\label{def-patch-U-k}
    \begin{aligned}
      U_0(T) & := T, \\
      U_k(T) & := \cup\{T'\in \T_H\;\vert\; T'\cap U_{k-1}(T)\neq\emptyset\}\quad k=1,2,\ldots .
    \end{aligned}
\end{equation}
\end{definition}
\noindent
The above concept of patch sizes and patch shapes can be also generalized. See for instance \cite{Henning:Morgenstern:Peterseim:2013} for a LOD that is purely based on partitions of unity.
Using Definition \ref{category-k}, we make an abstract assumption on the decay of the local correctors $Q_h^T(\Phi_H)$ for $\Phi_H \in V_H$:

\begin{itemize}
\item[(A8)] Let $Q_h^{\Omega,T}(\Phi_H)$ be the {\it optimal} local corrector using $U(T)=\Omega$ that is defined according to (\ref{local-corrector-problem}) and let $Q_h^{\Omega}(\Phi_H):=\sum_{T\in\T_H}Q_h^{\Omega,T}(\Phi_H)$.
Let $k\in \mathbb{N}_{>0}$ and for all $T \in \T_H$ let $U(T)=U_k(T)$ as in Definition \ref{category-k}. Then there exists $p\in \{0,1\}$ and a generic constant $0<\theta<1$ that can depend on the contrast, but not on $H$, $h$ or the variations of $A$ such that for all $\Phi_H \in V_H$,
\begin{eqnarray}
\label{equation-influence-intersections}||| (Q_h - Q_h^{\Omega})(\Phi_H) |||_h^2 \lesssim k^d
\theta^{2 k} (1/H)^{2p} ||| \Phi_H+ Q_h^{\Omega}(\Phi_H) |||_h^2,
\end{eqnarray}
where, $Q_h(\Phi_H)$ denotes the global corrector given by (\ref{global-corrector}) for $U(T)=U_k(T)$.
\end{itemize}

Assumption (A8) quantifies the decay of local correctors, by stating that the solutions of the local corrector problems decay exponentially to zero outside of $T$. This is central for all a priori error estimates. For continuous Galerkin methods, we can obtain the optimal order $p=0$ for the exponent in \eqref{equation-influence-intersections}. This means, that the $(1/H)$-term fully vanishes. However, depending on the localization strategy (i.e. how $Q_h(\Phi_H)$ is computed) it is also possible that $p$ takes the value $1$ and that hence a pollution term of order $(1/H)$ arises (see \cite[Remark 3.8]{Henning:Malqvist:2013} for a discussion). For discontinuous Galerkin methods, the optimal known order is $p=1$. However, even for this case it is known that the $(1/H)$-term is rapidly overtaken by the decay, leading purely to slightly larger patch sizes (see e.g. \cite{Malqvist:Peterseim:2011}).

\subsection{Galerkin LOD}
This method was originally proposed in \cite{Malqvist:Peterseim:2011}: find
$u_H^{\text{\tiny G-LOD}} \in \Vms$ that satisfies
\begin{eqnarray}
\label{lod-problem-eq} a_h(u_H^{\text{\tiny G-LOD}},\Phi^{\operatorname*{ms}}) = ( f , \Phi^{\operatorname*{ms}} )
\hspace*{0.3cm} \text{for all} \hspace*{0.3cm} \Phi^{\operatorname*{ms}} \in \Vms.
\end{eqnarray}

\begin{theorem}[A priori error estimate for Galerkin LOD]\label{t:a-priori-local}$\\$
Assume (A1)-(A8). Given a positive $k\in\mathbb{N}_{>0}$, let for all $T \in\T_H$ the patch $U(T)=U_k(T)$ be defined as in (\ref{def-patch-U-k}) and
let $u_H^{\text{\tiny \em{G-LOD}}} \in \Vms$ be as governed by \eqref{lod-problem-eq}. Let $u_h \in V_h$ be the fine scale reference solution governed by (\ref{fine-grid}). Then, the following a priori error estimate holds true
\begin{align*}
\| u_h - ((I_H\vert_{V_H})^{-1} \circ I_H)(u_H^{\text{\tiny \em{G-LOD}}}) \|_{L^2(\Omega)} + ||| u_h - u_H^{\text{\tiny \em{G-LOD}}} |||_h&\lesssim (H + (1/H)^{p} k^{d/2} \theta^{k} ) \|f\|_{L^2(\Omega)},
\end{align*}
where $0<\theta<1$ and $p\in\{0,1\}$ are the generic constants in (A8).
\end{theorem}

The term
$((I_H\vert_{V_H})^{-1} \circ I_H)(u_H^{\text{\tiny {G-LOD}}})$ describes the coarse part (resulting from $V_H$) of $u_H^{\text{\tiny {G-LOD}}}$ and
thus is numerically homogenized (the oscillations are averaged out). In this sense, we can say that $u_H^{\text{\tiny {G-LOD}}}$ is an $H^1$-approximation of $u_h$ and $((I_H\vert_{V_H})^{-1} \circ I_H)(u_H^{\text{\tiny {G-LOD}}})$ an $L^2$-approximation of $u_h$, respectively.
Furthermore, because
$k^{\frac{d}{2}} \theta^{k}$ converges with exponential order to zero, the error $||| u_h - u_H^{\text{\tiny G-LOD}} |||_h$ is typically dominated by the first term of order O$(H)$. This was observed in various numerical experiments in different works, c.f. \cite{Henning:Malqvist:2013,Henning:Malqvist:Peterseim:2013,Malqvist:Peterseim:2011}. In particular, a specific choice $k \gtrsim (p+1) |\log(H)|$ leads to a O$(H)$ convergence for the total $H^1$-error, see also \cite{Henning:Malqvist:2013,Henning:Malqvist:Peterseim:2013,Malqvist:Peterseim:2011}.

\subsection{Petrov-Galerkin LOD}
\label{subsection-pg-lod}

In a straightforward manner, we can now state the LOD in Petrov-Galerkin formulation:
find $u_H^{\text{\tiny PG-LOD}} \in \Vms$ that satisfies
\begin{eqnarray}
\label{lod-problem-eq-pg} a_h(u_H^{\text{\tiny PG-LOD}},\Phi_H ) = ( f , \Phi_H )
\hspace*{0.3cm} \text{for all} \hspace*{0.3cm} \Phi_H \in V_H.
\end{eqnarray}

A unique solution of (\ref{lod-problem-eq-pg}) is guaranteed by the inf-sup stability. In practice, inf-sup stability is clearly observable in numerical experiments (see Section \ref{section:numerical:experiments}). Analytically we can make the following observations.

\begin{remark}[Quasi-orthogonality and inf-sup stability]\label{quasi-orth}
The inf-sup stability of the LOD in Petrov-Galerkin formulation is a natural property to expect, since we have quasi-orthogonality in $a_h(\cdot,\cdot)$ of the spaces $\Vms$ and $W_h$. This can be verified by a simple computation. Let $\Phi^{\operatorname*{ms}}=\Phi_H + Q_h(\Phi_H) \in \Vms$, let $w_h \in W_h$ and let $Q_h^{\Omega}(\Phi_H)$ the optimal corrector as in assumption (A8), then
\begin{equation*}
\begin{aligned}
a_h ( \Phi^{\operatorname*{ms}} , w_h ) &= a_h ( \Phi_H + Q_h(\Phi_H) , w_h )\\
&= a_h ( Q_h(\Phi_H) - Q_h^{\Omega}(\Phi_H) , w_h ) \\
&\le ||| Q_h(\Phi_H) - Q_h^{\Omega}(\Phi_H) |||_h ||| w_h |||_h \\
&\lesssim k^{d/2} \theta^{k} (1/H)^{p} ||| \Phi_H + Q_h^{\Omega}(\Phi_H) |||_h ||| w_h |||_h,
\end{aligned}
\end{equation*}
with generic constants $0<\theta<1$ and $p\in \{ 0,1\}$ as in (A8). This means that $a_h ( \Phi^{\operatorname*{ms}} , w_h )$ converges exponentially in $k$ to zero, and it is identical to zero for all sufficiently large $k$ (because then $Q_h(\Phi_H)=Q_h^{\Omega}(\Phi_H)$). Writing the PG-LOD bilinear form as
\begin{eqnarray*}
\lefteqn{a_h( \Phi_H + Q_h(\Phi_H), \Psi_H )}\\
&=& a_h ( \Phi_H + Q_h(\Phi_H), \Psi_H + Q_h(\Psi_H) ) + a_h ( \Phi_H + Q_h(\Phi_H), Q_h(\Psi_H) ),
\end{eqnarray*}
we see that it is only a small perturbation of the symmetric (coercive) G-LOD version, where the difference can be bounded by the quasi-orthogonality.
\end{remark}

Even though the quasi-orthogonality {\it suggests} inf-sup stability, the given assumptions (A1)-(A8) do not seem to be sufficient for rigorously proving it. Here, it seems necessary to leave the abstract setting and to prove the inf-sup stability result for the various LOD realizations separately. For simplification, we therefore make the inf-sup stability to be an additional assumption (see (A9) below). Later we give an example how to prove this assumption for a certain realization of the method. We also note that the inf-sup stability can be always verified numerically (for a given $k$) by investigating the system matrix $S^{\text{\tiny PG-LOD}}$ given by the entries
\begin{align*}
(S^{\text{\tiny PG-LOD}})_{ij} = a_h( \Phi_j + Q_h(\Phi_j) , \Phi_i )
\end{align*}
for $1\le i,j \le N_H$ where $N_H$ denotes the dimension of $V_H$ and where $\{ \Phi_i| \hspace{2pt} 1 \le i \le N_H \}$ denotes a basis of $V_H$. To check the inf-sup stability we must compute the eigenvalues of $S^{\text{\tiny PG-LOD}}$. If their real parts are all strictly positive, we have inf-sup stability and the inf-sup constant is identical to the smallest real part of an eigenvalue. Standard approaches for computing the eigenvalues of a non-symmetric matrix are the Arnoldi method, the Jacobi-Davidson method and the non-symmetric Lanczos algorithm (cf. \cite{MR1893417} for a comprehensive overview). Since $N_H$ is moderately small, the cost for applying one of the methods are still feasible.

\begin{itemize}
\item[(A9)] We assume that the LOD in Petrov-Galerkin formulation is inf-sup stable in the following sense:
there exists a sequence of constants $\alpha(k)$ and a generic limit $\alpha_0>0$ (independent of $H$, $h$, $k$ or the oscillations of $A$) such that $\alpha(k)$ converges with {\it exponential speed} to $\alpha_0$, i.e. there exist constants $C(H)$ (possibly depending on $H$, but not on $h$, $k$ or the oscillations of $A$) and a generic $\theta \in (0,1)$ such that
$|\alpha(k)-\alpha_0| \le C(H) k^{d/2} \theta^k$. Furthermore it holds $\alpha(\bar{k})=\alpha_0$ for all sufficiently large $\bar{k}$ and 
\begin{align*}
\frac{a_h(\Phi^{\operatorname*{ms}},\Phi_H)}{||| \Phi_H |||_H} \ge \alpha(k) ||| \Phi^{\operatorname*{ms}} |||_h,
\end{align*}
for all $\Phi^{\operatorname*{ms}}\in \Vms$ and $\Phi_H:=((I_H\vert_{V_H})^{-1} \circ I_H)(\Phi^{\operatorname*{ms}})\in V_H$.
\end{itemize}

The following result states that the approximation quality of the LOD in Petrov-Galerkin formulation is of the same order as for the Galerkin LOD, up to a possible pollution term depending on $C_{H,h}$, but which still converges exponentially to zero.

\begin{theorem}[A priori error estimate for PG-LOD]\label{t:a-priori-local-pg}$\\$
Assume (A1)-(A9). Given a positive $k\in\mathbb{N}_{>0}$, let for all $T \in\T_H$ the patch $U(T)=U_k(T)$ be defined as in (\ref{def-patch-U-k}) and large enough so that the inf-sup constant in (A9) fulfills $\alpha(k)\ge \bar{\alpha}$ for some $\bar{\alpha}>0$ and let
$u_H^{\text{\tiny \em{PG-LOD}}}$ be the unique solution of (\ref{lod-problem-eq-pg}). 
 Let $u_h \in V_h$ be the fine scale reference solution governed by (\ref{fine-grid}).
  Then, the following a priori error estimate holds true
\begin{eqnarray*}
\lefteqn{\| u_h - ((I_H\vert_{V_H})^{-1} \circ I_H)(u_H^{\text{\tiny \em{PG-LOD}}}) \|_{L^2(\Omega)} + ||| u_h - u_H^{\text{\tiny \em{PG-LOD}}} |||_h} \\
&\lesssim& (H + (1/H)^{p}(1 + (1/\bar{\alpha}))(1+C_{H,h}) k^{d/2} \theta^{k} ) \|f\|_{L^2(\Omega)},
\end{eqnarray*}
where $0<\theta<1$ and $p \in \{0,1\}$ are the generic constants from assumption (A8) and $C_{H,h}$ as in (A5).
\end{theorem}

\subsection{Example 1: Continuous Galerkin Finite Element Method}
\label{subsection-pg-lod-example1}

The previous subsection showed that the Petrov-Galerkin formulation of the LOD does not suffer from a loss in accuracy with respect to the symmetric formulation. In this subsection, we give the specific example of the LOD for the Continuous Galerkin Finite Element Method. In particular, we discuss the advantage of the PG formulation over the symmetric formulation. Let us first introduce the specific setting and the corresponding argument about the validity of (A4)-(A9) on this setting.

In addition to the assumptions that we made on the shape regular partitions $\T_H$ and $\T_h$ in Section \ref{subsection:two-grid:disc}, we assume that $\T_H$ and $\T_h$ are either triangular or quadrilateral meshes. Accordingly, for $\T=\T_H,\T_h$ we denote
\begin{align*}
P_1(\T) &:= \{v \in C^0(\Omega) \;\vert \;\forall T\in\T,v\vert_T \text{ is a polynomial of total degree}\leq 1\} \quad \mbox{and}\\
Q_1(\T) &:= \{v \in C^0(\Omega) \;\vert \;\forall T\in\T,v\vert_T \text{ is a polynomial of partial degree}\leq 1\}
\end{align*}
and define
$V_h:=P_1(\T_h)\cap H^1_0(\Omega)$
if $\T_h$ is a triangulation and
$V_h:=Q_1(\T_h)\cap H^1_0(\Omega)$
if it is a quadrilation. The coarse space $V_H \subset V_h$ is defined in the same fashion and since $\T_h$ is a refinement of $\T_H$, assumption (A6) is obviously fulfilled. For simplicity, we also assume that the coarse mesh $\T_H$ is quasi-uniform (which is the typical choice in applications). 

The bilinear form $a_h(\cdot,\cdot)$ is defined by the standard energy scalar product on $H^1_0(\Omega)$ that belongs to the elliptic problem to solve, i.e.
\begin{align*}
a_h(v,w):= \int_{\Omega} A \nabla v \cdot \nabla w \qquad \mbox{for } v,w \in H^1_0(\Omega).
\end{align*}
Accordingly, we set $||| v |||_h:=||| v |||_H:=\| A^{1/2} \nabla v \|_{L^2(\Omega)}$ for $v \in H^1(\Omega)$. Hence, assumptions (A5) and (A6) are fulfilled and the solution $u_h \in V_h$ of (\ref{fine-grid}) is nothing but the standard continuous Galerkin Finite Element solution on the fine grid $\T_h$.

Next, we specify $I_H:V_h \rightarrow W_h$ in (A7).
For this purpose, let $\Phi_z \in V_H$ be the nodal basis function associated with the coarse grid node $z \in \mathcal{N}_H$, i.e.,
 $\Phi_z(y)=\delta_{yz}$. Let $I_H$ be the weighted Cl\'ement-type quasi-interpolation operator as defined in \cite{MR1736895,MR1706735}:
\begin{align}
\label{def-weighted-clement} I_H : H^1_0(\Omega) \rightarrow V_H,\quad v\mapsto I_H(v):= 
\sum_{z \in \mathcal{N}_{H}^0}
v_z \Phi_z \quad \text{with }v_z := \frac{(v,\Phi_z)_{L^2(\Omega)}}{(1,\Phi_z)_{L^2(\Omega)}}.
\end{align}
First we note that it was shown in \cite{Malqvist:Peterseim:2011} that $(I_H)_{|V_H}: V_H \rightarrow V_H$ is an isomorphism (but not a projection, i.e. $(I_H\vert_{V_H})^{-1} \neq I_H\vert_{V_H}$). Hence, $(I_H)_{|V_H}^{-1}$ exists.
This is one of the properties in (A7). The $L^2$- and $H^1$-stability of $I_H$, as well as corresponding approximation properties, were proved in \cite{MR1736895}. It only remains to check the $H^1$-stability of $(I_H)_{|V_H}^{-1}$. Unfortunately, this property is not trivial to fulfill. First, we note that it was shown in \cite{Malqvist:Peterseim:2012} that the mapping $(I_H)_{|V_H}^{-1} \circ I_H$ is nothing but the $L^2$-projection $P_{L^2} :H^1_0(\Omega) \rightarrow V_H$ (see also Remark \ref{remark-on-L2-projection} below). Consequently, the question of $H^1$-stability of $(I_H)_{|V_H}^{-1}$ is equivalent to the question of $H^1$-stability of the $L^2$-projection. This result is well-established for quasi uniform grids (cf. \cite{BaD81}) as assumed at the beginning of this section. However it is still open for arbitrary refinements. The most recent results on this issue can be found in \cite{BaY14,KPP13,GHS14}, where the desired $H^1$-stability was shown for certain types of adaptively refined meshes. To avoid complicated mesh assumptions in this paper, we simply assume $\T_H$ to be quasi-uniform. This is not very restrictive since adaptive refinements should typically take place on the fine mesh $\T_h$. Alternatively, in light of \cite{BaY14,KPP13,GHS14}, we could also directly assume that the $L^2$-projection on $V_H$ is $H^1$-stable to allow more general coarse meshes.

It remains to specify $a_h^T(\cdot,\cdot)$, which we define by
\begin{align*}
a_h^T(v,w):= \int_T A \nabla v \cdot \nabla w \qquad \mbox{for } v,w \in H^1_0(\Omega).
\end{align*}
Let us for simplicity denote $|||\cdot |||_{h,T}:=\| A^{1/2} \nabla \cdot \|_{L^2(T)}$. The decay assumption (A8) was essentially proved in \cite[Lemma 3.6]{Henning:Malqvist:2013}, which established the existence of a generic constant $0<\theta<1$ with the properties as in (A8) such that
\begin{eqnarray}
\label{equation-influence-intersections-cg}||| (Q_h - Q_h^{\Omega})(\Phi_H) |||_h^2 \lesssim k^d
\theta^{2 k}
\sum_{T\in\T_H} ||| Q_h^{\Omega,T}(\Phi_H) |||_h^2,
\end{eqnarray}
for all $\Phi_H \in V_H$ .
On the other hand we have by $|||\cdot |||_{h,T}=\| A^{1/2} \nabla \cdot \|_{L^2(T)}$ and equation (\ref{local-corrector-problem}) that
\begin{equation}
\label{local-energy}
\begin{aligned}
||| Q_h^{\Omega,T}(\Phi_H) |||_h^2 &\lesssim a_h( Q_h^{\Omega,T}(\Phi_H) , Q_h^{\Omega,T}(\Phi_H) ) \\
&= - a_h^T ( \Phi_H , Q_h^{\Omega,T}(\Phi_H) )\\
&\lesssim ||| \Phi_H |||_{h,T} ||| Q_h^{\Omega,T}(\Phi_H) |||_h.
\end{aligned}
\end{equation}
Hence, by plugging this result into \eqref{equation-influence-intersections-cg}:
\begin{eqnarray*}
||| (Q_h - Q_h^{\Omega})(\Phi_H) |||_h^2 &\lesssim& k^d \theta^{2 k} \sum_{T\in\T_H} ||| \Phi_H |||_{h,T}^2 \\
&\lesssim&
k^d \theta^{2 k} ||| \Phi_H |||_h^2 = k^d \theta^{2 k} ||| ((I_H\vert_{V_H})^{-1} \circ I_H) (\Phi_H+Q_h^{\Omega}(\Phi_H)) |||_h^2 \\
&\overset{(A7)}{\lesssim} & k^d \theta^{2 k} ||| \Phi_H+Q_h^{\Omega}(\Phi_H) |||_h^2,
\end{eqnarray*}
which proves that assumption (A8) holds even with $p=0$. The remaining assumption (A9) is less obvious and requires a proof.  We give this proof 
for the Continuous Galerkin PG-LOD in Section \ref{section:proofs:pg:lod}. We summarize the result in the following lemma.

\begin{lemma}[inf-sup stability of Continuous Galerkin PG-LOD]
\label{inf-sup-stability-pg-lod}
For all $T \in \T_H$ let $U(T)=U_{k}(T)$ for $k \in \mathbb{N}$. Then there exist generic constants $C_1,C_2$ (independent of $H$, $h$, $k$ or the oscillations of $A$) and $0 < \theta < 1$ as in assumption (A8), so that it holds
\begin{align*}
\inf_{\Phi_H \in V_H} \sup_{\Phi^{\operatorname*{ms}} \in \Vms} \frac{a(\Phi^{\operatorname*{ms}},\Phi_H)}{|||  \Phi^{\operatorname*{ms}} |||_h ||| \Phi_H |||_h} \ge \alpha(k),
\end{align*}
for $\alpha(k) := C_1 \alpha - C_2 k \theta^{k} \omega(\Phi^{\operatorname*{ms}})$ and
\begin{align*}
0 \le \omega(\Phi^{\operatorname*{ms}}) := \inf_{w_h \in W_h^T} \frac{\| \nabla \Phi^{\operatorname*{ms}} - \nabla ((I_H\vert_{V_H})^{-1} \circ I_H)( \Phi^{\operatorname*{ms}} ) - \nabla w_h \|}{\|  \nabla \Phi^{\operatorname*{ms}} - \nabla ((I_H\vert_{V_H})^{-1} \circ I_H)( \Phi^{\operatorname*{ms}} ) \|} \le 1,
\end{align*}
where $W_h^T:= \{ w_h \in W_h | \hspace{2pt} w_h\vert_{T} \in W_h(T) \}$, i.e. the space of all functions from $W_h$ that are zero on the boundary of the coarse grid elements. Observe that $\alpha(k)$ converges with exponential speed to $\alpha C_1$. Furthermore we have $\alpha(0)=C_1 \alpha$ (because $\omega(\Phi^{\operatorname*{ms}})=0$) and also $\alpha(\ell)=C_1 \alpha$ for all sufficiently large $\ell$.
\end{lemma}

\begin{remark}
Let $U(T)=U_{k}(T)$ for $k \in \mathbb{N}$ with $k \gtrsim |\log(H)|$, then the CG-LOD in Petrov-Galerkin formulation is inf-sup stable for sufficiently small $H$. In particular, there exists a unique solution of problem (\ref{lod-problem-eq-pg}).
\end{remark}

\begin{remark}
Lemma \ref{inf-sup-stability-pg-lod} does not allow to conclude to inf-sup stability for the regime $0<k\ll|\log(H)|$. However, even though this regime is not of practical relevance, it is interesting to note that we could not observe a violation of the inf-sup stability for any value of $k$ and in any numerical experiment that we set up so far.
\end{remark}

Since assumptions (A1)-(A9) are fulfilled for this setting, Theorems \ref{t:a-priori-local} and \ref{t:a-priori-local-pg} hold true for the arising method. Furthermore, we have $p=0$ and $C_{H,h}=1$ in the estimates, meaning that the $(1/H)$-pollution in front of the decay term vanishes. We can summarize the result in the following conclusion.

\begin{conclusion}
Assume the (Continuous Galerkin) setting of this subsection and let $u_H^{\text{\tiny \em{PG-LOD}}}$ denote a Petrov-Galerkin solution of (\ref{lod-problem-eq-pg}). If $k \gtrsim m H |\log(H)|$ for $m \in \mathbb{N}$, then it holds
\begin{eqnarray*}
\| u_h - u_H^{\text{\tiny \em{PG-LOD}}} \|_{H^1(\Omega)} &\lesssim& (H + H^m ) \|f\|_{L^2(\Omega)}.
\end{eqnarray*}
In particular, the bound is independent of $C_{H,h}$.
\end{conclusion}

\subsection{Discussion of advantages}
The central disadvantage of the Galerkin LOD is that it requires a communication between solutions of different patches. Consider for instance the assembly of the system matrix that belongs to problem (\ref{lod-problem-eq}). Here it is necessary to compute entries of the type
\begin{align*}
\int_{\Omega} A \nabla (\Phi_i + Q_h(\Phi_i) ) \cdot \nabla (\Phi_j + Q_h(\Phi_j) ),
\end{align*}
which particularly involves the computation of the term
\begin{align}
\label{problematic-term}\underset{T \subset \omega_i}{\sum_{T \in \T_H}} \underset{K \subset \omega_j}{\sum_{K \in \T_H}} 
\int_{U(T) \cap U(K)} A \nabla Q_h^T(\Phi_i) \cdot \nabla Q_h^K(\Phi_j),
\end{align}
where $\Phi_i,\Phi_j \in V_H$ denote two coarse nodal basis functions and $\omega_i$ and $\omega_j$ its corresponding supports. The efficient computation of (\ref{problematic-term}) requires information about the intersection area of any two patches $U(T)$ and $U(K)$. Even if $T$ and $K$ are not adjacent or close to each other, the intersection of the corresponding patches can be complicated and non-empty. The drawback becomes obvious: first, these intersection areas must be determined, stored and handled in an efficient way and second, the number of relevant entries of the stiffness matrix (i.e. the non-zeros) increases considerably. Note that this also leads to a restriction in the parallelization capabilities, in the sense that the assembly of the stiffness matrix can only be 'started' if the correctors $Q_h(\Phi_i)$ are already computed. Another disadvantage is that the assembly of the right hand side vector associated with $(f,\Phi^{\operatorname*{ms}})$ in (\ref{lod-problem-eq}) is much more expensive since it involves the computation of entries $(f,\Phi_i+Q_h(\Phi_i))_{L^2(\Omega)}$. First, the integration area is $\cup 
\{ U(T) | \hspace{2pt} T \in \T_H, \hspace{2pt} T \subset \omega_i\}$ instead of typically $\omega_i$. This increases the computational costs. At the same time, it is also hard to assemble these entries by performing (typically more efficient) element-wise computations (for which each coarse element has to be visited only once). Second, $(f,\Phi_i+Q_h(\Phi_i))_{L^2(\Omega)}$ involves a quadrature rule of high order, since $Q_h(\Phi_i)$ is rapidly oscillating. These oscillations must be resolved by the quadrature rule, even if $f$ is a purely macroscopic function that can be handled exactly by a low order quadrature. Hence, the costs for computing $(f,\Phi_i+Q_h(\Phi_i))_{L^2(\Omega)}$ depend indirectly on the oscillations of $A$. Finally, if the LOD shall be applied to a sequence of problems of type (\ref{exact-solution}), which only differ in the source term $f$ (or a boundary condition), the system matrix can be fully reused, but the complications that come with the right hand side have to be addressed each time again.

The Petrov-Galerkin formulation of the LOD clearly solves these problems without suffering from a loss in accuracy. In particular:
\begin{itemize}
\item The PG-LOD does not require any communication between two different patches and the resulting stiffness matrix is sparser than the one for the symmetric LOD. In particular, the entries of the system matrix $S$ can be computed with the following algorithm:\\
\begin{algo}
\label{algorithm-computing-stiffness-pg-lod}
 \rule{0.9\textwidth}{.7pt} \\
Let $S$ denote the empty system matrix with entries $S_{ij}$.
 \rule{0.9\textwidth}{.7pt} \\
Algorithm: assembleSystemMatrix( $\T_H$, $\T_h$, $k$ )
 \rule{0.9\textwidth}{.7pt} \\
In parallel \ForEach{$T \in \T_H$}
{
\ForEach{$z_i \in \mathcal{N}_H^0$ \mbox{\rm with} $z_i \in \overline{T}$}
 {
   compute $Q_h^T(\Phi_{z_i}) \in W_h(U_k(T))$ \mbox{\rm with}
\begin{align*}
 a(Q_h^T(\Phi_{z_i}),w_h)=-\int_{T} A \nabla \Phi_{z_i} \cdot \nabla w_h  \quad\text{for all }w_h\in W_h(U_k(T)).
\end{align*}
\ForEach{$z_j \in \mathcal{N}_H^0$ \mbox{\rm with} $z_j \in \overline{U(T)}$}
{
update the system matrix:
\begin{align*}
S_{ji} \hspace{3pt} +\hspace{-3pt}= \int_{\omega_j} A \left( \Phi_{z_i} + \nabla Q_h^T( \Phi_{z_i} ) \right) \cdot \nabla \Phi_{z_j}.
\end{align*}
}
 }
}
 \rule{0.9\textwidth}{.7pt}
\end{algo}\\
Observe that it is possible to directly add the local terms $a(\Phi_{z_i} + Q_h^T( \Phi_{z_i} ),  \Phi_{z_j})$ to the system matrix $S$, i.e. the assembling of the matrix is parallelized in a straightforward way and does not rely on the availability of other results.
\item Replacing the source term $f$ in (\ref{exact-solution}), only involves the re-computation of the terms $(f,\Phi_i)_{L^2(\omega_i)}$ for coarse nodal basis functions $\Phi_i$, i.e. the same costs as for the standard FE method on the coarse scale. Furthermore, the choice of the quadrature rule relies purely on $f$, but not on the oscillations of $A$.
\end{itemize}

Besides the previously mentioned advantages, there is still a memory consuming issue left: the storage of the local correctors $Q_h^T( \Phi_{z_i} )$. These local correctors need to be saved in order to express the final approximation $u_H^{\text{\tiny PG-LOD}}$ which is spanned by the multiscale basis functions $\Phi_i + Q_h(\Phi_i)$. As long as we are interested in a good $H^1$-approximation of the solution, this problem seems to be unavoidable. However, in many applications we can even overcome this difficulty by exploiting another very big advantage of the PG-LOD: Theorem \ref{t:a-priori-local-pg} predicts that alone the 'coarse part' of $u_H^{\text{\tiny PG-LOD}}$, denoted by $u_H:=((I_H\vert_{V_H})^{-1} \circ I_H)(u_H^{\text{\tiny PG-LOD}})) \in V_H$, already exhibits very good $L^2$-approximation properties, i.e. if $k \gtrsim |\log(H)|$ we have essentially
\begin{align*}
\| u_h - u_H \|_{L^2(\Omega)} \leq \mbox{O}(H).
\end{align*}
In contrast to $u_H^{\text{\tiny PG-LOD}}$, the representation of $u_H$ does only require the classical coarse finite element basis functions. Hence, we can use the  algorithm presented earlier, with the difference that we can immediately delete $Q_h^T(\Phi_i)$ after updating the stiffness matrix. Observe that even if computations have to be repeated for different source terms $f$, this stiffness matrix can be reused again and again. Also, if a user is interested in the fine scale behavior in a local region (but the $Q_h^T(\Phi_i)$ were already dropped), it is still possible to quickly re-compute the desired local corrector for the region.

As an application, consider for instance the case that the problem
\begin{align*}
\int_{\Omega} A \nabla u \cdot \nabla v = \int_{\Omega} f v
\end{align*}
describes the diffusion of a pollutant in groundwater. Here, $u$ describes the concentration of the pollutant, $A$ the (rapidly varying) hydraulic conductivity and $f$ a source term describing the injection of the pollutant. In such a scenario, there is typically not much interest in finding a good approximation of the (locally fluctuating) gradient $\nabla u$, but rather in the macroscopic behavior of pollutant $u$, i.e. in purely finding a good $L^2$-approximation that allows to conclude where the pollutant spreads. A similar scenario is the investigation of the properties of a composite material, where $A$ describes the heterogenous material and $f$ some external force. Again, the interest is in finding an accurate $L^2$-approximation. Besides, the corresponding simulations are typically performed for a variety of different source terms $f$, investigating different scenarios. In this case, the PG-LOD yields reliable approximations with very low costs, independent of the structure of $A$.

\begin{remark}[Relation to the $L^2$-projection]
\label{remark-on-L2-projection}Assume the setting of this subsection. In \cite{Malqvist:Peterseim:2012} it was shown that $(v_H,w_h)_{L^2(\Omega)}=0$ for all $v_H\in V_H$ and $w_h\in W_h$, i.e. $V_H$ and $W_h$ are $L^2$-orthogonal. This implies that
$$(I_H\vert_{V_H})^{-1} \circ I_H = P_{L^2},$$
with $P_{L^2}$ denoting the $L^2$-projection on $V_H$. To verify this, let $v_h \in V_h$ be arbitrary. Then due to $V_h = V_H \oplus W_h$ we can write $v_h=v_H+w_h$ (with $v_H \in V_H$ and $w_h \in W_h$) and observe for all $\Phi_H \in V_H$
\begin{align*}
 \int_{\Omega} P_{L^2}(v_h) \hspace{2pt} \Phi_H &= \int_{\Omega} v_h \hspace{2pt} \Phi_H \overset{V_H \perp_{L^2} W_h}{=} \int_{\Omega} v_H \hspace{2pt} \Phi_H \\
 &= \int_{\Omega} ((I_H\vert_{V_H})^{-1} \circ I_H)(v_H) \hspace{2pt} \Phi_H 
 \overset{I_H(w_h)=0}{=} \int_{\Omega} ((I_H\vert_{V_H})^{-1} \circ I_H)(v_h) \hspace{2pt} \Phi_H.
\end{align*}
Hence, $u_H^{\text{\tiny PG-LOD}}=u_H + Q_h(u_H)$ with $u_H=P_{L^2}(u_H^{\text{\tiny PG-LOD}})$.
\end{remark}

\begin{conclusion}[Application to homogenization problems]
\label{homogenization-setting}Assume the setting of this subsection and let $P_{L^2}$ denote the $L^2$-projection on $V_H$ as in Remark \ref{remark-on-L2-projection}. We consider now a typical homogenization setting with $(\epsilon)_{>0}\subset \R_{>0}$ being a sequence of positive parameters that converges to zero. Let $Y:=[0,1]^d$ denote the unique cube in $\R^d$ and let $A^{\epsilon}(x)=A_p(x,\frac{x}{\epsilon})$ for a function $A_p \in W^{1,\infty}(\Omega \times Y)$ that is $Y$-periodic in the second argument (hence $A^{\epsilon}$ is rapidly oscillating with frequency $\epsilon$). The corresponding exact solution of problem (\ref{exact-solution}) shall be denoted by $u_{\epsilon} \in H^1_0(\Omega)$. It is well known (c.f. \cite{All92}) that $u_{\epsilon}$ converges weakly in $H^1$ (but not strongly) to some unique function $u_0 \in H^1_0(\Omega)$. Furthermore, if $\|f\|_{L^2(\Omega)}\lesssim 1$ it holds $\| u_{\epsilon} - u_0 \|_{L^2(\Omega)} \lesssim \epsilon$. With Theorem \ref{t:a-priori-local-pg} together with Remark \ref{remark-on-L2-projection} and standard error estimates for FE problems, we hence obtain:
\begin{align*}
\| u_0 - u_H \|_{L^2(\Omega)} \lesssim \epsilon + \left(\frac{h}{\epsilon}\right)^2 + H,
\end{align*}
for $u_H=P_{L^2}(u_H^{\text{\tiny PG-LOD}})$. Homogenization problems are typical problems, where one is often purely interested in the $L^2$-approximation of the exact solution $u_{\epsilon}$, meaning one is interested in the homogenized solution $u_0$.
\end{conclusion}

As discussed in this section, the PG-LOD can have significant advantages over the (symmetric) G-LOD with respect to computational costs, efficiency and memory demand. In Subsection \ref{num-experiment-cg-pg-lod} we additionally present a numerical experiment to demonstrate that the approximations produced by the PG-LOD are in fact very close to the ones produced by (symmetric) G-LOD, i.e. not only of the same order as predicted by the theorems, but also of the same quality.

\begin{remark}[Nonlinear problems]
The above results suggest that the advantages can become even more pronounced for certain types of nonlinear problems. For instance, consider a well-posed problem of the type
\begin{align*}
- \nabla \cdot A \nabla u + c(u) = f,
\end{align*} 
for a nonlinear function $c$. Here, it is intuitively reasonable to construct $Q_h(\Phi_H)$ as before using only the linear elliptic part of the problem.
This is a preprocessing step that is done once and can be immediately deleted stiffness matrix is calculated and saved. Then we solve for $u_H \in V_H$ that satisfies
\begin{align*}
( A \nabla (u_H + Q_h(u_H)), \nabla \Phi_H )_{L^2(\Omega)} + ( c(u_H) ,\Phi_H )_{L^2(\Omega)} = ( f ,\Phi_H )_{L^2(\Omega)} \quad \mbox{for all } \Phi_H\in V_H.
\end{align*}
Clearly, typical iterative solvers can be utilized to solve this variational problem. This iteration is inexpensive because it is done in $V_H$ and the preconstructed stiffness matrix can be fully reused within every iteration and since the other contributions are independent of $Q_h$. Performing iterations
on the coarse space for solving nonlinear problems within the framework of multiscale finite element (MsFEMs) has been investigated (see
for example \cite{MR2062582} and \cite{Efendiev:Hou:Ginting:2004}). 
\end{remark}

\subsection{Example 2: Discontinuous Galerkin Finite Element Method}
\label{subsection-pg-dg-lod}

In this subsection, we apply the results of Section \ref{subsection-pg-lod} to a LOD Method that is based on a Discontinuous Galerkin approach. The DG-LOD was originally proposed in \cite{EGMP13}
and fits into the framework proposed in Section \ref{subsection:lod-framework}. First, we show that the setting fulfills assumptions (A4)-(A8) and after we discuss the advantage of the PG DG-LOD over the symmetric DG-LOD. For simplification, we assume that $A$ is piecewise constant with respect to the fine mesh $\T_h$ so that all of the subsequent traces are well-defined.

Again, we make the same assumptions on the partitions $\T_H$ and $\T_h$ as in Section \ref{subsection:two-grid:disc} and  additionally assume that $\T_H$ and $\T_h$ are either triangular or quadrilateral meshes. The corresponding total sets of edges (or faces for $d=3$) are denoted by $\Epsilon_h$ (for $\T_h$), where $\Epsilon_h(\Omega)$ and $\Epsilon_h(\partial\Omega)$ denotes the set of interior and boundary edges, respectively.

Furthermore, for $\T=\T_H,\T_h$ we denote the spaces of discontinuous functions with total, respectively partial, polynomial degree equal to or less than $1$ by
\begin{align*}
\mathcal{P}_1(\T) &:= \{v \in L^2\Omega) \;\vert \;\forall T\in\T,v\vert_T \text{ is a polynomial of total degree}\leq 1\} \quad \mbox{and}\\
\mathcal{Q}_1(\T) &:= \{v \in L^2(\Omega) \;\vert \;\forall T\in\T,v\vert_T \text{ is a polynomial of partial degree}\leq 1\}
\end{align*}
and define
$V_h:=P_1(\T_h)$ if $\T_h$ is a triangulation and $V_h:=Q_1(\T_h)$
if it is a quadrilation. The coarse space $V_H \subset V_h$ is defined in the same fashion with $\T_H$ instead of $\T_h$. Note that these spaces are no subspaces of $H^1(\Omega)$ as in the previous example. For this purpose, we define $\nabla_h$ to be the $\T_h$-piecewise gradient (i.e. $(\nabla_h v_h)\vert t := \nabla (v_h\vert t )$ for $v_h \in V_h$ and $t \in \T_h$).

For every edge/face $e \in \Epsilon_h(\Omega)$ there are two adjacent elements $t^-,t^+ \in \T_h$ with $e=\partial t^- \cap \partial t^+$. We define the jump and average operators across $e\in\Epsilon_h(\Omega)$ by
\begin{align*}
[v] &:= (v\vert t^-  - v\vert t^+)\quad\text{and}\quad\{ A \nabla v \cdot n \} := \frac{1}{2}( (A \nabla v)\vert t^- + (A \nabla v)\vert t^+) \cdot n,
\end{align*}
where $n$ be the unit normal on $e$ that points from $t^-$ to $t^+$, and on $e\in\Epsilon_h(\partial\Omega)$ by
\begin{align*}
[v] &:= w\vert t\quad\text{and}\quad \{ A \nabla v \cdot n \} :=  (A \nabla v)\vert t  \cdot n
\end{align*}
where $n$ is the outwards unit normal of $t\in \T_h$ (and $\Omega$). Observe that flipping the roles of $t^-$ and $t^+$ leads to the same terms in the bilinear form defined below.

With that, we can define the typical bilinear form that characterizes the Discontinuous Galerkin method:
\begin{eqnarray*}
\lefteqn{a_h( v_h, w_h ) := (A \nabla_h v_h, \nabla_h w_h)_{L^2(\Omega)}} \\
&\enspace& - \sum_{e \in \Epsilon_h} \left( 
(\{ A \nabla v_h \cdot n \},  [w_h] )_{L^2(e)} + (\{ A \nabla w_h \cdot n \},  [v_h] )_{L^2(e)} \right)
+ \sum_{e \in \Epsilon_h} \frac{\sigma}{h_e} ( [v_h],  [w_h] )_{L^2(e)}.
\end{eqnarray*}
Here, $\sigma$ is a penalty parameter that is chosen sufficiently large and $h_e=\text{diam}(e)$. The coarse bilinear form $a_H(\cdot,\cdot)$ is defined analogously with coarse scale quantities. It is well known, that $a_h(\cdot,\cdot)$ (respectively $a_H(\cdot,\cdot)$) is a scalar product on $V_h$ (respectively $V_H$). Consequently (A4) is fulfilled. As a norm on $V_h$ that fulfills assumption (A5), we can pick
\begin{align*}
||| v |||_h := \| A^{1/2} \nabla_h v \|_{L^2(\Omega)} +  \left( \sum_{e \in \Epsilon_h} \frac{\sigma}{h_e} \| [v] \|^2_{L^2(e)} \right)^{1/2}.
\end{align*}
Analogously, we define $||| v |||_H$ to be a norm on $V_H$. In this case we obtain the constant $C_{H,h}=\sqrt{H/h}$. Assumption (A6) is obviously fulfilled.

As the operator in assumption (A7) we pick the $L^2$-projection on $V_H$, i.e. for $v_h \in V_h$ we have
\begin{align*}
(I_h(v_h),\Phi_H)_{L^2(\Omega)} = (v_h,\Phi_H)_{L^2(\Omega)} \qquad \mbox{for all } \Phi_H \in V_H.
\end{align*} 
In \cite[Lemma 5]{EGMP13} it was proved that the operator fulfills the desired approximation and stability properties. Since $I_H$ is a projection, we have $I_H = ( I_H\vert_{V_H})^{-1}$ and hence obviously also $|||\cdot|||_H$-stability of the inverse on $V_H$.

The localized bilinear form $a_h^T( \cdot , \cdot )$ in (\ref{localization-of-a_h-to-a_h-T}) is defined by $a_h^T( v_h, w_h ) := a_h( \chi_Tv_h, w_h )$ where $\chi_T=1$ in $T$ and $0$ otherwise, is the element indicator function. Obviously we have for all $v_h,w_h \in V_h$ that 
\begin{align*}
a_h ( v_h , w_h ) = \sum_{T \in \T_H} a_h^T( v_h, w_h ).
\end{align*}

In \cite{EGMP13} the DG-LOD is presented in a slightly different way, in the sense that there exists no general corrector operator $Q_h$. Instead, 'basis function correctors' are introduced. However, it is easily checkable that each of these 'basis function correctors' is nothing but the corrector operator, defined via (\ref{local-corrector-problem}), applied to an original coarse basis function. Therefore, the correctors given by (\ref{local-corrector-problem}) are just an extension of the definition to arbitrary coarse functions. Hence, both methods coincide and are just presented in a different way.

Next, we discuss (A8). This property was shown in \cite[Lemma 11 and 12]{EGMP13}, however not explicitly for the setting that we established in Definition \ref{localization-of-solution-space}. It was only shown for $\Phi_H=\lambda_{T,j}$, where $\lambda_{T,j} \in V_H$ denotes a basis function on $T$ associated with the $j$'th node. However, the proofs in \cite{EGMP13} directly generalize to the local correctors $Q_h^T(\Phi_H)$ given by equation \eqref{local-corrector-problem}. More precisely, following the proofs in \cite{EGMP13} it becomes evident that the availability of the required decay property (A8) purely relies on the fact, that the right hand side in the local problems is only locally supported (with a support that remains fixed, even if the patch size decreases). Therefore (A8) can be proved analogously.

Finally, assumption (A9) is not easy to verify. It is obviously fulfilled for the case $U(T)=\Omega$, but the generalized result is harder to verify. The following result holds under some restrictions on the meshes $\T_H$ and $\T_h$.
\begin{lemma}[inf-sup stability of Discontinuous Galerkin PG-LOD]
\label{inf-sup-stability-dg-pg-lod}
Assume that $\T_H$ is quasi-uniform and that there exists an exponent $m \in \mathbb{R}$ with $m > 1$ such that for all $T\in \T_H$ 
$$\mbox{\rm diam}(T)^m \lesssim
\min \{ h_e | \hspace{2pt} e \in \mathcal{E}_h  \enspace \mbox{and} \enspace e \subset \overline{T} \}$$
(i.e. if $\T_h$ is also quasi-uniform we assume $H^m \lesssim h$). If $k \in \mathbb{N}$ is such that $k \gtrsim \frac{(m+3)}{2} |\log(H)|$ then, for sufficiently small $H$, there exist generic positive constants $C_1,C_2$ such that \begin{align*}
\inf_{\Phi_H \in V_H} \sup_{\Phi^{\operatorname*{ms}}\in \Vms} \frac{a_h(\Phi^{\operatorname*{ms}},\Phi_H)}{||| \Phi^{\operatorname*{ms}} |||_h||| \Phi_H |||_H} \ge C_1( \alpha - C_2 H).
\end{align*}
Hence, we have inf-sup stability for sufficiently small $H$.
\end{lemma}
The proof is given in Section \ref{section:proofs:pg:lod}. We note that the inf-sup stability can be observed numerically already under weaker assumptions (see Section \ref{section:numerical:experiments}) and that it is in general 'a reasonable thing to expect' as discussed in Remark \ref{quasi-orth}.

In conclusion, the Discontinuous Galerkin LOD in Petrov-Galerkin formulation fulfills the assumptions of our framework (up to a discussion on (A9)). The advantages that we discussed in the previous subsection for the Petrov-Galerkin Continuous Finite Element Method in terms of memory and efficiency
remains true. However, for the PG DG-LOD there is a very important additional advantage. It is known that the classical DG method has the feature of local mass conservation with respect to the elements of the underlying mesh. This can be easily checked by testing with the indicator function of an element $T$ in the variational formulation of the method. The local mass conservation is a highly desired property for various flow and transport problems. However, the DG-LOD does not preserve this property, since the indicator function of an element (whether coarse or fine) is not in the space $\Vms$. This problem is solved in the PG DG-LOD, where we can test with any element from $V_H$ and in particular with the indicator function of a coarse element. Hence, in contrast to the symmetric DG-LOD, the PG DG-LOD is locally mass conservative with respect to coarse elements $T \in \T_H$. This allows for example the coupling of the PG DG-LOD for an elliptic problem with the solver for a hyperbolic conservation law, which was not possible before without relinquishing the mass conservation. We discuss this further in the next subsection.

\subsection{Perspectives towards Two-Phase flow}
\label{ssec:tpflow}
In this subsection, we investigate an application of the Petrov-Galerkin DG-LOD in the simulation of two-phase flow as governed by the Buckley-Leverett equation. Specifically, the LOD framework is utilized to solve the pressure equation, which is an elliptic boundary value problem, and is coupled with 
a solver for a hyperbolic conservation law. The Buckley-Leverett equation can be used to model two-phase flow in a porous medium. Generally, the flow of two immiscible and incompressible fluids is driven by the law of mass balance for the two fluids:
\begin{align}
\label{tpf-system}{\Theta} \partial_t S_{\alpha} + \nabla \cdot \bfs{v}_{\alpha} &= q_{\alpha} \qquad \mbox{in } \Omega \times (0,T_{end}] \quad \mbox{for } \alpha=w,n.
\end{align}
Here, $\Omega$ is a computational domain, $(0,T_{end}]$ a time interval, the unknowns $S_{w},S_{n}:\Omega \rightarrow [0,1]$ describe the saturations of a wetting and a non-wetting fluid and $\bfs{v}_{w}$ and $\bfs{v}_{n}$ are the corresponding fluxes. Furthermore, $\Theta$ describes the porosity and $q_{w}$ and $q_n$ are two source terms. Darcy's law relates the fluxes with the two unknown pressures $p_n$ and $p_w$ by
\begin{align*}
\bfs{v}_{\alpha} = - K \frac{k_{\alpha}(S_{\alpha})}{\mu_{\alpha}} (\nabla p_{\alpha} - \rho_{\alpha} \bfs{g}) \qquad \mbox{for} \enspace \alpha = w,n.
\end{align*}
Here, $K$ denotes the hydraulic conductivity, $k_w$ and $k_n$ the relative permeabilities depending on the saturations, $\mu_w$ and $\mu_n$ the viscosities, $\rho_w$ and $\rho_n$ the densities and $\bfs{g}$ the gravity vector. The saturations are coupled via $S_n+S_w=1$ and a relation between the two pressures is typically given by the capillary pressure relation $P_c(S_w)=p_n-p_w$ for a monotonically decreasing capillary pressure curve $P_c$. In this case, we obtain the full two-phase flow system, which consists of two strongly coupled, possibly degenerate parabolic equations. However, if we neglect the gravity and the capillary pressure (i.e. assume that $P_c(S_w)=0$), the system reduces to the so called Buckley-Leverett system with an elliptic pressure equation and an hyperbolic equation for the saturation:
\begin{align}
\label{bl-system}-\nabla \cdot \left( K \lambda(S) \nabla p \right) &= q \quad \mbox{and} \quad
\Theta \partial_t S + \nabla \cdot \left( f(S) \bfs{v} \right) = q_{w},
\end{align}
where we have $S=S_w$, $p=p_w=p_n$, the total mobility $\lambda(S):=\frac{k_{w}(S)}{\mu_{w}}+\frac{k_{n}(1-S)}{\mu_{n}}>0$, the flux $\bfs{v}:=-K \lambda(S) \nabla p$ and the flux function $f(S):=\frac{k_{w}(S)}{\mu_{w} \lambda(S)}$. The total source is given by $q:=\frac{q_w+q_n}{2}$. Observe that (\ref{bl-system}) is obtained from (\ref{tpf-system}) by summing up the equations for the saturations, using $\partial_t(s_n+s_w)= \partial_t 1 = 0$.

An application for which neglecting the capillary pressure is typically justified are oil recovery processes. Here, a replacement fluid, such as water or liquid carbon dioxide, is injected with very high rates into a reservoir to move oil towards a production well. However, often oil is trapped at interfaces of a low and a high conductivity region. This oil would become inaccessible which is why detailed simulations are required before the replacement fluid can be actually injected.

Depending on how the mobilities are chosen, the hyperbolic Buckley-Leverett problem can have one or more weak solutions (c.f. \cite{LeFloch:2002}). One approach for solving the problem numerically is to use an operator splitting technique as proposed in \cite{Aziz:Settari:1979}, which is more well-known as the \emph{(IM)plicit (P)ressure (E)xplicit (S)aturation}, i.e., IMPES. Here, the hyperbolic Buckley-Leverett problem is treated with an explicit time stepping method where the flux velocity $\bfs{v}$ is kept constant for a certain time interval and then updated by solving the elliptic problem with the saturation from the previous time step (see Figure~\ref{fig:os-impes} for an illustration). Alternatively, depending on the type of the flux function $f$, the hyperbolic problem can be also solved implicitly with a suitable numerical scheme for conservation laws (c.f. \cite{Kroener:1997}) where the flux $\bfs{v}$ arising from the Darcy equation is, as in the previous case, only updated every fixed number of time steps.

\begin{figure}
\centering
\includegraphics[scale=0.6]{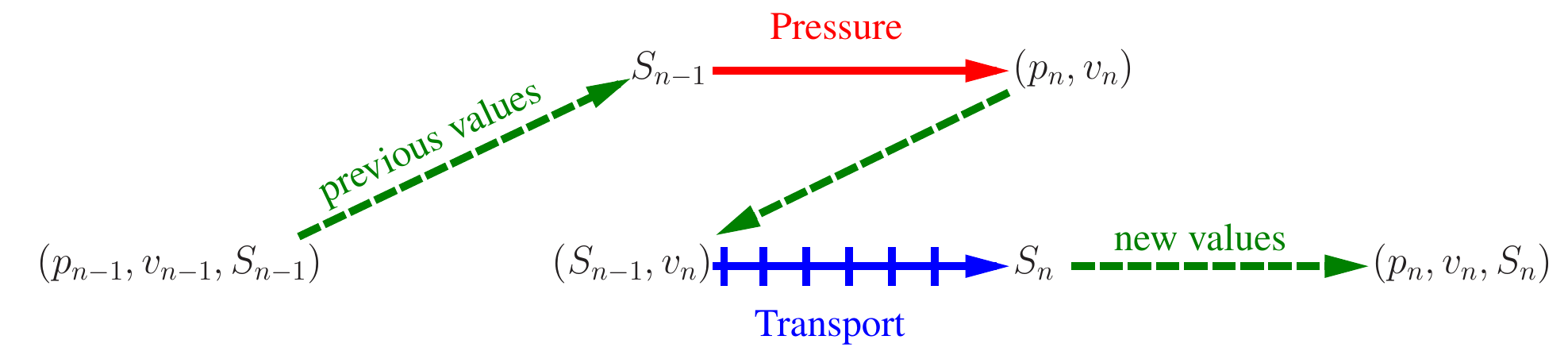}
\caption{A schematic of operator splitting (IMPES) for system \eqref{bl-system}}
\label{fig:os-impes}
\end{figure}

Observe that the difficulties produced by the multiscale character of the problem are primarily related to the elliptic part of the problem. Once the Darcy problem is solved to update the flux velocity, the grid for solving the hyperbolic problem can be significantly coarsened. The reason is that $\bfs{v}=-K \lambda(S) \nabla p$ is possibly still rapidly oscillating, but the relative amplitude of the oscillations is expected to remain small. In other words, just like for standard elliptic homogenization problems, $\bfs{v}$ behaves like an upscaled quantity $-K_0 \lambda(S_0) \nabla p_0$ with effective/homogenized functions $K_0$, $S_0$ and $p_0$.

\begin{remark}
Any realization of the LOD involves to solve a number of local problems that help us to construct the low dimensional space $\Vms$. One might consider to update this space every time that the Darcy problem has to be solved with a new saturation. However, since $\lambda(S)$ is essentially macroscopic, it is generally sufficient to construct the space only once for $\lambda=1$ and reuse the result for every time step. This makes solving the elliptic multiscale problem much cheaper after the multiscale space is assembled. A justification for this reusing of the basis can be e.g. found in \cite{Henning:Malqvist:Peterseim:2013} where it was shown that oscillations coming from advective terms can be often neglected in the construction of a multiscale basis. Under certain assumptions, the relative permeability $\lambda(S)$ can in fact be interpreted as a pure enforcement by an additional advection term. 
\end{remark}

\section{Numerical Experiments}
\label{section:numerical:experiments}
\setcounter{equation}{0}
In this section we present two different model problems. The first one involves a LOD methods for the continuous Galerkin method. Here, we compare the results obtained with the symmetric version of the method with the results obtained for the Petrov-Galerkin version. In the second model problem, we use a PG DG-LOD for solving the Buckley-Leverett system.

\subsection{Continuous Galerkin PG-LOD for elliptic multiscale problems}
\label{num-experiment-cg-pg-lod}

In this section, we use the setting established in Section \ref{subsection-pg-lod-example1}. All experiments were performed with the G-LOD and PG-LOD for the Continuous Finite Element Method. 

\begin{figure}[!ht]
\centering
\includegraphics[width=0.7\textwidth]{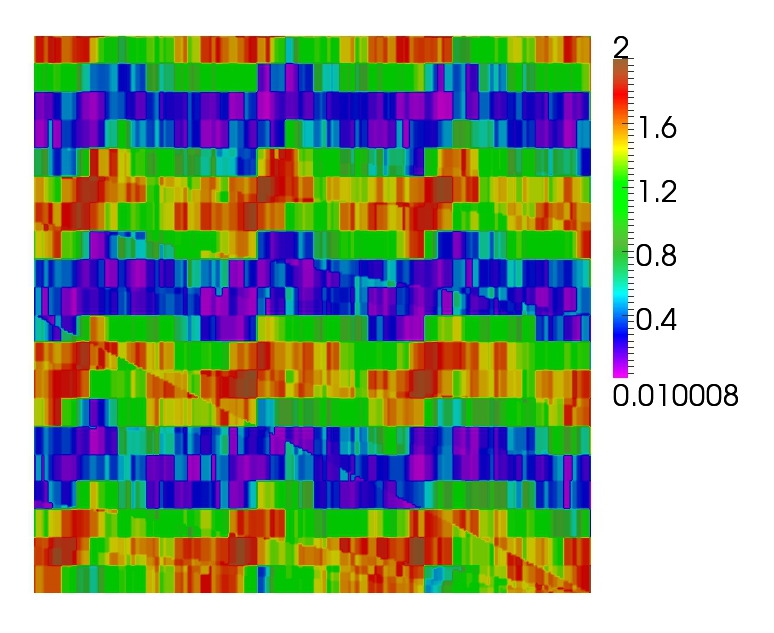}
\caption{Sketch of heterogeneous diffusion coefficient $A_\varepsilon$ defined according to equation (\ref{diffusion-coefficient}).}
\label{diffusion_problem_1}
\end{figure}

\begin{table}[t]
\caption{Results for the errors between LOD approximations and reference solutions. We define $e_h:= u_h -u^{\text{\tiny G-LOD}}$ and $e_h^{\mbox{\tiny PG}}:= u_h -u^{\text{\tiny PG-LOD}}$. Accordingly we define the errors between the reference solution and the coarse parts of the LOD approximations by $e_H:= u_h -P_{L^2}(u^{\text{\tiny G-LOD}})$ (for the symmetric case) and $e_{H}^{\mbox{\tiny PG}}:= u_h -P_{L^2}(u^{\text{\tiny PG-LOD}})$ (for the Petrov-Galerkin case). The reference solution $u_h$ was obtained on a fine grid of mesh size $h=2^{-6}\approx 0.0157 < \varepsilon$ which just resolves the micro structure of the coefficient $A_\varepsilon$. The number of 'coarse grid layers' is denoted by $k$ and determines the patch size $U_k(T)$.}
\label{table-layers-results}
\begin{center}
\begin{tabular}{|c|c|c|c|c|c|c|c|c|}
\hline $H$      & $k$
& {\color{dark-red}$\| e_H \|_{L^2(\Omega)}^{\mbox{\tiny rel}}$}
& {\color{dark-red}$\| e_h \|_{L^2(\Omega)}^{\mbox{\tiny rel}}$}
& {\color{dark-red}$\| e_h \|_{H^1(\Omega)}^{\mbox{\tiny rel}}$}
& {\color{dark-blue}$\| e_H^{\mbox{\tiny PG}} \|_{L^2(\Omega)}^{\mbox{\tiny rel}}$}
& {\color{dark-blue}$\| e_h^{\mbox{\tiny PG}} \|_{L^2(\Omega)}^{\mbox{\tiny rel}}$}
& {\color{dark-blue}$\| e_h^{\mbox{\tiny PG}} \|_{H^1(\Omega)}^{\mbox{\tiny rel}}$} \\
\hline
\hline $2^{-2}$ & 0    & {\color{dark-red}0.3794} & {\color{dark-red}0.3772} & {\color{dark-red}0.6377}
                                 & {\color{dark-blue}0.3778} & {\color{dark-blue}0.3755} & {\color{dark-blue}0.6375} \\
\hline $2^{-2}$ & 1/2 & {\color{dark-red}0.2756} & {\color{dark-red}0.2381} & {\color{dark-red}0.5312}
                                 & {\color{dark-blue}0.2588} & {\color{dark-blue}0.2269} & {\color{dark-blue}0.5628} \\
\hline $2^{-2}$ & 1    & {\color{dark-red}0.2523} & {\color{dark-red}0.1445} & {\color{dark-red}0.3637}
                                 & {\color{dark-blue}0.2544} & {\color{dark-blue}0.1504} & {\color{dark-blue}0.3642} \\
\hline $2^{-2}$ & 3/2 & {\color{dark-red}0.2514} & {\color{dark-red}0.1355} & {\color{dark-red}0.3125}
                                 & {\color{dark-blue}0.2518} & {\color{dark-blue}0.1380} & {\color{dark-blue}0.3162} \\
\hline
\hline $2^{-3}$ & 0   & {\color{dark-red}0.2039} & {\color{dark-red}0.2037}  & {\color{dark-red}0.5048}
                                & {\color{dark-blue}0.2037} & {\color{dark-blue}0.2036}  & {\color{dark-blue}0.5048} \\
\hline $2^{-3}$ & 1   & {\color{dark-red}0.1100} & {\color{dark-red}0.0526}  & {\color{dark-red}0.2278}
                                & {\color{dark-blue}0.1139}  & {\color{dark-blue}0.0619} & {\color{dark-blue}0.2345} \\
\hline $2^{-3}$ & 2   & {\color{dark-red}0.1073} & {\color{dark-red}0.0423} & {\color{dark-red}0.1761} 
                                & {\color{dark-blue}0.1078} & {\color{dark-blue}0.0453} & {\color{dark-blue}0.1807} \\
\hline $2^{-3}$ & 3   & {\color{dark-red}0.1070} & {\color{dark-red}0.0366} & {\color{dark-red}0.1567}
                                & {\color{dark-blue}0.1077} & {\color{dark-blue}0.0399} & {\color{dark-blue}0.1600} \\
\hline
\hline $2^{-4}$ & 0   & {\color{dark-red}0.0874} & {\color{dark-red}0.0873} & {\color{dark-red}0.3563}
                                & {\color{dark-blue}0.0874} & {\color{dark-blue}0.0873} & {\color{dark-blue}0.3563} \\
\hline $2^{-4}$ & 2   & {\color{dark-red}0.0353} & {\color{dark-red}0.0105} & {\color{dark-red}0.0932}
                                & {\color{dark-blue}0.0357} & {\color{dark-blue}0.0123} & {\color{dark-blue}0.0994} \\
\hline $2^{-4}$ & 4   & {\color{dark-red}0.0351} & {\color{dark-red}0.0082} & {\color{dark-red}0.0653}
                                & {\color{dark-blue}0.0353} & {\color{dark-blue}0.0093} & {\color{dark-blue}0.0680} \\
\hline $2^{-4}$ & 6   & {\color{dark-red}0.0351} & {\color{dark-red}0.0080} & {\color{dark-red}0.0634}
                                & {\color{dark-blue}0.0353} & {\color{dark-blue}0.0091} & {\color{dark-blue}0.0662} \\
\hline
\end{tabular}\end{center}
\end{table}

\begin{table}[t]
\caption{Results for the errors between LOD approximations and reference solutions. The errors are defined as in Table \ref{table-layers-results}. The reference solution $u_h$ was obtained on a fine grid of mesh size $h=2^{-8}\approx 0.0039 \ll \varepsilon$ which fully resolves the micro structure of the coefficient $A_\varepsilon$. Again, the number of 'coarse grid layers' is denoted by $k$ and determines the patch size $U_k(T)$.}
\label{table-layers-results-2}
\begin{center}
\begin{tabular}{|c|c|c|c|c|c|c|c|c|}
\hline $H$      & $k$
& {\color{dark-red}$\| e_H \|_{L^2(\Omega)}^{\mbox{\tiny rel}}$}
& {\color{dark-red}$\| e_h \|_{L^2(\Omega)}^{\mbox{\tiny rel}}$}
& {\color{dark-red}$\| e_h \|_{H^1(\Omega)}^{\mbox{\tiny rel}}$}
& {\color{dark-blue}$\| e_H^{\mbox{\tiny PG}} \|_{L^2(\Omega)}^{\mbox{\tiny rel}}$}
& {\color{dark-blue}$\| e_h^{\mbox{\tiny PG}} \|_{L^2(\Omega)}^{\mbox{\tiny rel}}$}
& {\color{dark-blue}$\| e_h^{\mbox{\tiny PG}} \|_{H^1(\Omega)}^{\mbox{\tiny rel}}$} \\
\hline
\hline $2^{-2}$ & 0     & {\color{dark-red}0.3840} & {\color{dark-red}0.3815} & {\color{dark-red}0.6434}
                                  & {\color{dark-blue}0.3820} & {\color{dark-blue}0.3796} & {\color{dark-blue}0.6432} \\
\hline $2^{-2}$ & 1/8  & {\color{dark-red}0.2985} & {\color{dark-red}0.2781} & {\color{dark-red}0.5486}
                                  & {\color{dark-blue}0.2957} & {\color{dark-blue}0.2753} & {\color{dark-blue}0.5513} \\
\hline $2^{-2}$ & 1/4  & {\color{dark-red}0.2852} & {\color{dark-red}0.2592} & {\color{dark-red}0.5578}
                                  & {\color{dark-blue}0.2718} & {\color{dark-blue}0.2472} & {\color{dark-blue}0.5774} \\
\hline $2^{-2}$ & 1/2  & {\color{dark-red}0.2769} & {\color{dark-red}0.2392} & {\color{dark-red}0.5386}
                                  & {\color{dark-blue}0.2607} & {\color{dark-blue}0.2291} & {\color{dark-blue}0.5722} \\
\hline $2^{-2}$ & 3/4  & {\color{dark-red}0.2676} & {\color{dark-red}0.2052} & {\color{dark-red}0.4784}
                                  & {\color{dark-blue}0.2577} & {\color{dark-blue}0.1972} & {\color{dark-blue}0.4956} \\
\hline
\hline $2^{-3}$ & 0     & {\color{dark-red}0.2106} & {\color{dark-red}0.2103} & {\color{dark-red}0.5190}
                                  & {\color{dark-blue}0.2103} & {\color{dark-blue}0.2100} & {\color{dark-blue}0.5190} \\
\hline $2^{-3}$ & 1/4  & {\color{dark-red}0.1480} & {\color{dark-red}0.1375} & {\color{dark-red}0.4510}
                                  & {\color{dark-blue}0.1569} & {\color{dark-blue}0.1469} & {\color{dark-blue}0.4486} \\
\hline $2^{-3}$ & 1/2  & {\color{dark-red}0.1372} & {\color{dark-red}0.1163} & {\color{dark-red}0.3957} 
                                  & {\color{dark-blue}0.1305} & {\color{dark-blue}0.1089} & {\color{dark-blue}0.4029} \\
\hline $2^{-3}$ & 1     & {\color{dark-red}0.1138} & {\color{dark-red}0.0535} & {\color{dark-red}0.2308}
                                  & {\color{dark-blue}0.1176} & {\color{dark-blue}0.0628} & {\color{dark-blue}0.2372} \\
\hline $2^{-3}$ & 3/2  & {\color{dark-red}0.1117} & {\color{dark-red}0.0399} & {\color{dark-red}0.1710}
                                  & {\color{dark-blue}0.1126} & {\color{dark-blue}0.0437} & {\color{dark-blue}0.1761} \\
\hline
\hline $2^{-4}$ & 0     & {\color{dark-red}0.0988} & {\color{dark-red}0.0984} & {\color{dark-red}0.3854}
                                  & {\color{dark-blue}0.0987} & {\color{dark-blue}0.0983} & {\color{dark-blue}0.3854} \\
\hline $2^{-4}$ & 1/2  & {\color{dark-red}0.0637} & {\color{dark-red}0.0592} & {\color{dark-red}0.2896}
                                  & {\color{dark-blue}0.0500} & {\color{dark-blue}0.0442} & {\color{dark-blue}0.2934} \\
\hline $2^{-4}$ & 1     & {\color{dark-red}0.0406} & {\color{dark-red}0.0211} & {\color{dark-red}0.1613}
                                  & {\color{dark-blue}0.0431} & {\color{dark-blue}0.0263} & {\color{dark-blue}0.1690} \\
\hline $2^{-4}$ & 2     & {\color{dark-red}0.0381} & {\color{dark-red}0.0109} & {\color{dark-red}0.0957}
                                  & {\color{dark-blue}0.0385} & {\color{dark-blue}0.0130} & {\color{dark-blue}0.1017} \\
\hline $2^{-4}$ & 3     & {\color{dark-red}0.0380} & {\color{dark-red}0.0087} & {\color{dark-red}0.0726}
                                  & {\color{dark-blue}0.0382} & {\color{dark-blue}0.0099} & {\color{dark-blue}0.0753} \\
\hline
\end{tabular}\end{center}
\end{table}

\begin{figure}
\centering
\includegraphics[scale=0.6]{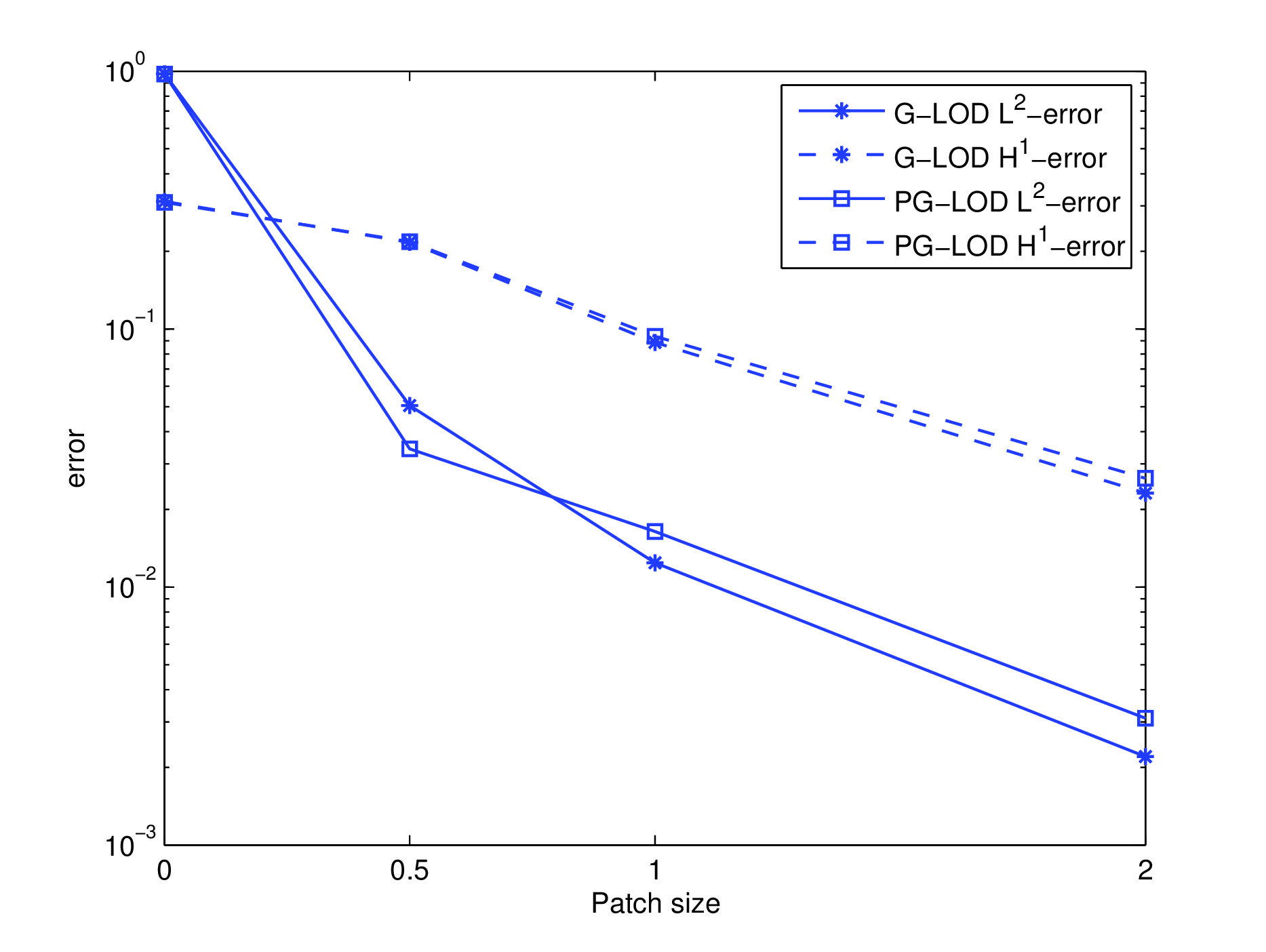}
\caption{The graphic visualizes the error decay in $k$. The results correspond to the results of Table \ref{table-layers-results-2} for $(h,H)=(2^{-8},2^{-4})$. We include $\| e_h \|_{L^2(\Omega)}^{\mbox{\tiny rel}}$, $\| e_h \|_{H^1(\Omega)}^{\mbox{\tiny rel}}$, $\| e_h^{\mbox{\tiny PG}} \|_{L^2(\Omega)}^{\mbox{\tiny rel}}$ and $\| e_h^{\mbox{\tiny PG}} \|_{H^1(\Omega)}^{\mbox{\tiny rel}}$. The $x$-axis depicts the localization parameter $k$ and the $y$-axis the error "$\| e(k) \| - \| e(3) \|$" on the $\log$-scale, where $\| e(k) \|$ denotes an error for $k$-layers (the error $\| e(3) \|$ is hence the limit reference).}
\label{convLayerPic}
\end{figure}

\begin{figure}[!ht]
\centering
\includegraphics[width=1.0\textwidth]{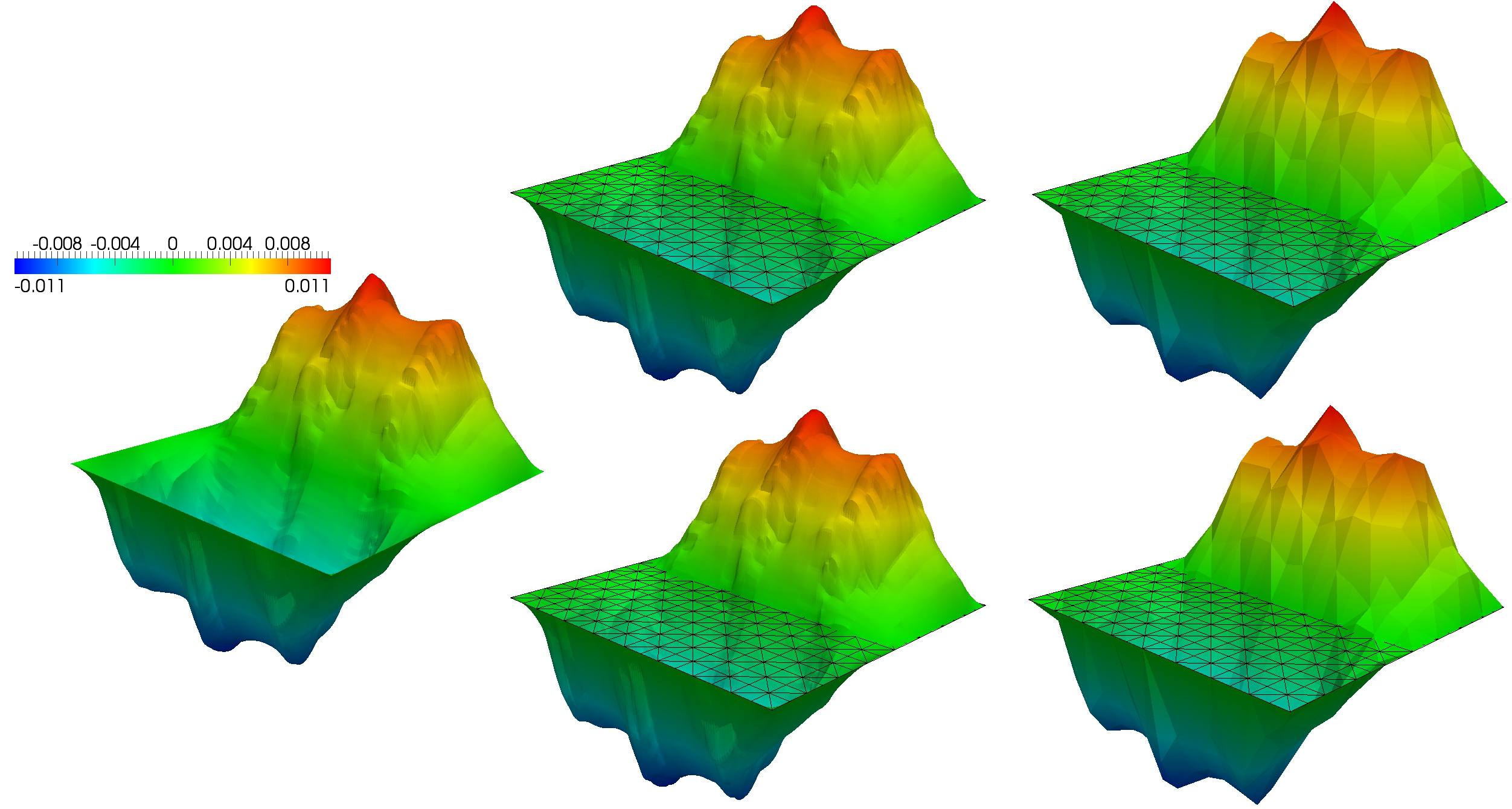}
\caption{The left picture shows the finite element reference solution $u_h$ for $h=2^{-8}$. The remaining pictures show LOD approximations for the case $(H,k)=(2^{-4},2)$, where $k$ denotes the (broken) number of coarse layers. The two top row pictures show the full G-LOD approximation $u^{\text{\tiny G-LOD}}$ (left) and the coarse part of it, i.e. $P_{L^2}(u^{\text{\tiny G-LOD}})$ (right). The bottom row shows the full Petrov-Galerkin LOD approximation $u^{\text{\tiny PG-LOD}}$ (left) and the corresponding coarse part, i.e. $P_{L^2}(u^{\text{\tiny PG-LOD}})$ (right). The grid that is added to each of the pictures shows the coarse grid $\T_H$.}
\label{series-lod-warp}
\end{figure}

\begin{figure}[!ht]
\centering
\includegraphics[width=0.4\textwidth]{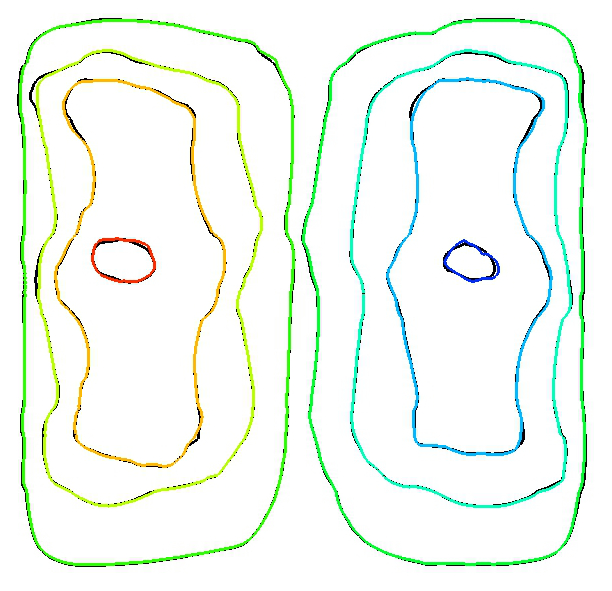}\hspace{20pt}
\includegraphics[width=0.5\textwidth]{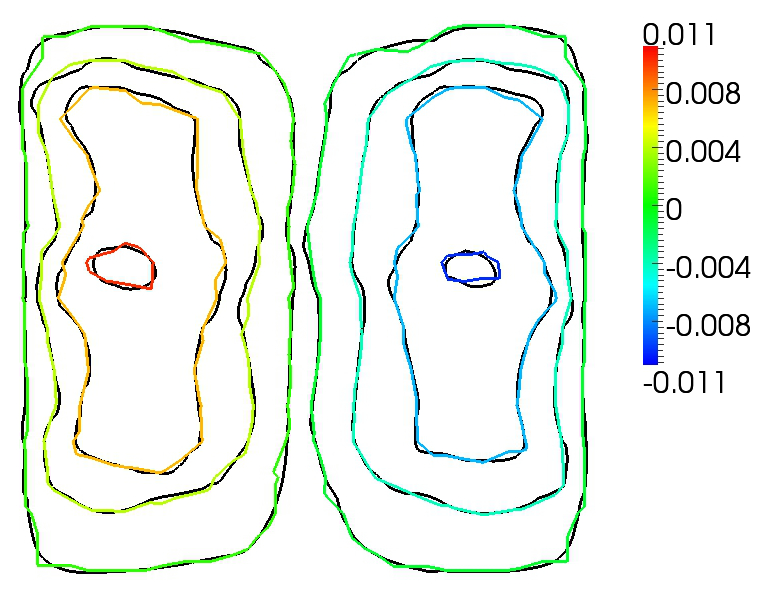}
\caption{The pictures depict a comparison of isolines. The black lines belong to the reference solution $u_h$ for $h=2^{-8}$. The colored isolines in the left picture belong to the PG-LOD approximation $u^{\text{\tiny PG-LOD}}$ and match almost perfectly with the one from the reference solution. The right picture shows the coarse part of $u^{\text{\tiny PG-LOD}}$, i.e. $P_{L^2}(u^{\text{\tiny PG-LOD}})$. We observe that the isolines still match nicely.}
\label{series-lod-isolines}
\end{figure}

In order to be more flexible in the choice of the localization patches $U(T)$, we make subsequently use of "half" or "quarter coarse layers", i.e. $k \in \mathbb{Q}_{\ge 0}$. This can be easily accomplished by extending Definition \ref{category-k} straightforwardly to fine grid layers, i.e. for $k\in \mathbb{Q}_{\ge0}$ and $T \in \T_H$ we define the number of fine layers by $\ell := \lfloor\frac{kH}{h}\rfloor \in \mathbb{N}$ and the corresponding (broken layer) patch by 
$U_k(T) := U_{\text{f},\ell}(T)$,
where iteratively 
$U_{\text{f},\ell}(T) := \cup \overline{\{t\in \T_h\;\vert\; t\cap U_{\text{f},\ell-1}(T)\neq\emptyset\}}$
and
$U_{\text{f},0}(T) := \overline{T}$. This allows us a more careful investigation of the decay behavior.

Let $u_h$ be the solution of (\ref{fine-grid}). In the following we denote by $\| \cdot \|_{L^2(\Omega)}^{\mbox{\tiny rel}}$ and $\| \cdot \|_{H^1(\Omega)}^{\mbox{\tiny rel}}$ the corresponding relative error norms defined by
\begin{align*}
\| u_h - v_h \|_{L^2(\Omega)}^{\mbox{\tiny rel}}:= \frac{\| u_h - v_h \|_{L^2(\Omega)}}{\| u_h\|_{L^2(\Omega)}} \quad \mbox{and} \quad \| u_h - v_h \|_{H^1(\Omega)}^{\mbox{\tiny rel}}:= \frac{\| u_h - v_h \|_{H^1(\Omega)}}{\| u_h\|_{H^1(\Omega)}} 
\end{align*}
for any $v_h \in V_h$. The coarse part ('the $V_H$-part') of an LOD approximation $u^{\text{\tiny G-LOD}}$ (respectively $u^{\text{\tiny PG-LOD}}$) is subsequently denoted by $P_{L^2}(u^{\text{\tiny G-LOD}})$ (respectively $P_{L^2}(u^{\text{\tiny PG-LOD}})$), where $P_{L^2}$ denotes the $L^2$-projection on $V_H$ (see also Remark \ref{remark-on-L2-projection}).

We consider the following model problem. Let $\Omega := \left]0,1\right[^2$ and $\varepsilon:=0.05$. Find $u_\varepsilon\in H^1(\Omega)$ with
\begin{align*}
- \nabla \cdot \left( A_\varepsilon(x) \nabla u_\varepsilon(x) \right) 
&= x_1 - \frac{1}{2} \qquad \text{in } \Omega \\
 u_\varepsilon(x) &= 0 \qquad \hspace{28pt} \text{on } \partial \Omega.
\end{align*}
The scalar diffusion term $A_\varepsilon$ is shown in Figure \ref{diffusion_problem_1}. It is given by
\begin{align}
\label{diffusion-coefficient}A_\varepsilon(x) := (h \circ c_\varepsilon)(x)
\qquad \text{with} \enspace h(t):=\begin{cases}
t^4 &\text{for} \enspace \frac12 < t < 1 \\ 
t^{\frac{3}{2}} &\text{for} \enspace 1 < t < \frac{3}{2}  \\ 
t &\text{else}
\end{cases}
\end{align}
and where
\begin{displaymath}
c_\varepsilon(x_1,x_2):=1 + \frac{1}{10} \sum_{j=0}^4 \sum_{i=0}^{j} \left( \frac{2}{j+1} \cos \left( \bigl\lfloor i x_2 - \tfrac{x_1}{1+i} \bigr\rfloor + \left\lfloor \tfrac{i x_1}\varepsilon \right\rfloor + \left\lfloor \tfrac{ x_2}\varepsilon \right\rfloor \right) \right).
\end{displaymath}

The goal of the experiments is to investigate the accuracy of the PG-LOD, compared to the classical symmetric LOD. Moreover, we investigate the accuracy of the coarse part of the LOD approximation in terms of $L^2$-approximation properties (see Section \ref{subsection-pg-lod-example1} for a corresponding discussion).

In Table \ref{table-layers-results} we can see the results for a fine grid $\T_h$ with resolution $h=2^{-6} < \varepsilon$ which just resolves the micro structure of the coefficient $A_\varepsilon$. Comparing the relative $L^2$- and $H^1$-errors for the G-LOD and the PG-LOD respectively (with the reference solution $u_h$), we observe that the errors are of similar size in each case. In general, we obtain slightly worse results for the Petrov-Galerkin LOD, however the difference is so small that is does not justify the usage of the more memory-demanding (and more expensive) symmetric LOD. For both methods we observe the same nice error decay (in terms of the patch size) that was already predicted by the theoretical results. Comparing the relative $L^2$-errors between $u_h$ and the coarse parts of the LOD-approximations, we observe that they already yield very good approximations. We also observe that they seem to be much more dominated by $H$-error contribution than by the $\theta^{k}$-error contribution (i.e. the error coming from the decay). Using patches consisting of more than $8$ fine element layers did not lead to any significant improvement, while there were still clear improvements visible for the other errors for the full G-LOD approximations. Furthermore, the linear convergence in $H$ is clearly visible for $\| e_H \|_{L^2(\Omega)}^{\mbox{\tiny rel}}$ (respectively $\| e_H^{\mbox{\tiny PG}} \|_{L^2(\Omega)}^{\mbox{\tiny rel}}$) showing that the obtained error estimates seem to be indeed optimal.

The same observations can be made for the errors depicted in Table \ref{table-layers-results-2} for a fine grid $\T_h$ with resolution $h=2^{-8} \ll \varepsilon$. Again, the results for the (symmetric) G-LOD are slightly better than the ones for the PG-LOD, but always of the same order. The exponential convergence in $k$ for both realization is visualized in Figure \ref{convLayerPic}. It is clearly observable that there is no argument for using the G-LOD when dealing with patch communication issues which are storage demanding.

These findings are confirmed in the Figures \ref{series-lod-warp} and \ref{series-lod-isolines}. In Figure \ref{series-lod-warp} we can see a visual comparison of the reference solution with the corresponding full LOD approximations (symmetric and Petrov-Galerkin). Both are almost not distinguishable for the investigated setting with $(h,H,k)=(2^{-8},2^{-4},2)$. Also the coarse parts of the LOD approximations already capture all the essential behavior of the reference solution. In Figures \ref{series-lod-isolines} this is emphasized. Here, we compare the isolines between the reference solution and PG-LOD approximation (respectively its coarse part) and we observe that they are highly matching.

\subsection{PG DG-LOD for the Buckley-Leverett equation}
In this subsection we present the results of a two-phase flow
simulation, based on solving the Buckley-Leverett equation as
discussed in subsection~\ref{ssec:tpflow}.  Recall that, the
Buckley-Leverett equation has two parts, a hyperbolic equation for the
saturation and a elliptic equation for the pressure. For that reason,
we use the operator splitting technique IMPES, that we stated in
subection~\ref{ssec:tpflow}. The elliptic pressure equation is solved
by the PG DG-LOD for which a discontinuous linear finite element
method is utilized that allows for recovering an elemental locally
conservative normal flux. We emphasize that having a locally
conservative flux is typically central for numerical schemes for
solving hyperbolic partial differential equations. In this experiment
we use an upwinding scheme.

Employing PG DG-LOD in this simulation proves to be a very efficient
since the local correctors for the generalized basis functions only
have to be computed once in a preprocessing step, this follows from
the fact the saturation only influence the permeability on the
macroscopic scale. The time stepping in the IMPES scheme using the PG
DG-LOD for the pressure is realized through the following algorithm.

\begin{algo}
\label{algorithm-bl-pg-dg-lod}
\rule{0.9\textwidth}{.7pt}\\ Set the end time $T_{end}$, number of update of
the pressure $n$, number of explicit updates on each implicit step update $m$.
\rule{0.9\textwidth}{.7pt} \\
Algorithm 2: solveBuckleyLeverett($\T_H$, $\T_h$, $T_{\text{end}}$, $n$, $m$)
\rule{0.9\textwidth}{.7pt} \\
Set the initial values: $S = S_0$ and $ i = 1$ \\
Preprocessing step: Compute local corrections $Q^T_h$ for all $T\in\T_H$ with $\lambda(S)=1$\\
\While{$t \leq T_\text{end}$} { Compute pressure $p$ using PG
  DG-LOD at $\left(t +
    T_\text{end}/(n)\right)$\\
  Extract conservative flux $\mathbf{v}$\\
  \While{$t \leq i T_\text{end}/n$} {
    Compute saturation $S$ at $\left(t + T_\text{end}/(n m)\right)$ \\
    Update time: $t + T_\text{end}/(n m) \mapsto t$ \\
  }
  $i + 1 \mapsto i$\\
} \rule{0.9\textwidth}{.7pt}
\end{algo}
\vspace{10pt}

In the numerical experiment we consider the domain $\Omega$ to be the
unit square. The permeability $K_i$ for $i=1,2$ is given by layer
21 and 31 of the Society of Petroleum Engineering comparative
permeability data (\texttt{http://www.spe.org/web/csp}), projected on a
uniform mesh with resolution $2^{-6}$ as illustrated in Figure~\ref{fig:SPE}.
\begin{figure}[htb]
\centering
\includegraphics[scale=0.8]{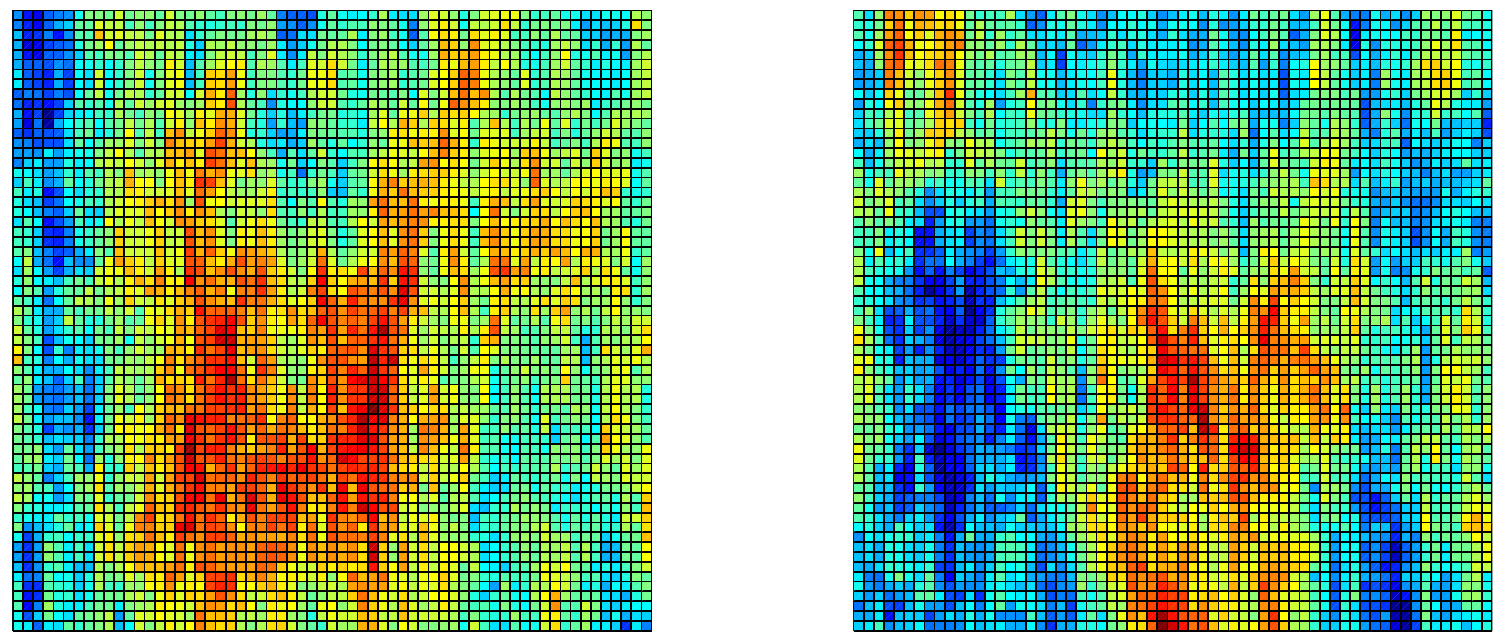}
\caption{The permeability structure of $K_i$ in log scale with,
  $\beta_0/\alpha_0 \approx 5\cdot 10^5$
  for $i=1$ (left) and
  $\beta_0/\alpha_0 \approx 4\cdot 10^5$ for
  $i=2$ (right).}
\label{fig:SPE}
\end{figure}
We consider a microscopic partition $\T_h$ with mesh size size
$h=2^{-8}$ and a macroscopic partition $\T_H$ with mesh size
$H=2^{-i}$ for $i=3,4,5,6$. The patch size is chosen
such that the overall $H$ convergence for the PG DG-LOD is not
effected. A reference solution to the Buckley-Leverett equation is
obtained when both the pressure and saturation equation are computed
on $\T_h$, compared to using Algorithm~\ref{algorithm-bl-pg-dg-lod}
where both the pressure and saturation equation are computed on
$\T_H$. We consider the following setup. For the pressure equation we
use the boundary condition $p=1$ for the left boundary, $p=0$ for the
right boundary, $K\lambda(S)\nabla p=0$ otherwise, and the source
terms $q_w=q_n=0$. For the saturation the initial value is $S=1$ on
the left boundary and $0$ elsewhere. The error is defined by
$e(\cdot,t) := S(\cdot,t)-S^{\mathrm{rel}}(\cdot,t)$, where
$S(\cdot,t)$ is the solution obtained by
Algorithm~\ref{algorithm-bl-pg-dg-lod} (at time $t$) and
$S^{\mathrm{rel}}(\cdot,t)$ is the reference solution (at time
$t$). The errors are measured in the $L^2$-norm. In
  Table~\ref{table-BL-results} we fix the coarse mesh size to be
  $H=2^{-5}$, and compute the error for the permeabilities $K_1$ and
  $K_2$ at the times $T_1:=0.05$, $T_2:=0.25$ and $T_3:=0.45$. A
graphical comparison is shown in Figure~\ref{fig:saturation21} and
\ref{fig:saturation31}. The errors in the $L^2$-norm is less than
$0.1$ for both permeabilities at all times which is quite remarkable
since the coarse mesh $\T_H$ for $H=2^{-5}$ does not
resolve the data. In Table~\ref{table-BL-convergence} we
consider the test case involving permeability $K_1$. We present the $L^2$-errors
at $t=T_2$ for different values of $H$. We basically observe a linear convergence rate in $H/h$
(for fixed $h$) which is just what we would expect (since we only use the coarse part of the LOD pressure approximation).

\begin{table}[ht]
  \caption{The resulting error in relative $L^2$-norm between $S$ and $S^{\mathrm{ref}}$, where $S$ is obtained using PG DG-LOD for the pressure computed on $\T_H$ and $S^{\mathrm{ref}}$ is the reference solution computed on $\T_h$. We have $T_1=0.05$, $T_2=0.25$ and $T_3=0.45$.}
\label{table-BL-results}
\begin{center}
  \begin{tabular}{|c|c|c|c|} \hline
    Data & $\|e(T_1)\|_{L^2(\Omega)}$ & $\|e(T_2)\|_{L^2(\Omega)}$ & $\|e(T_3)\|_{L^2(\Omega)}$ \\ \hline\hline
    $K_1$ & 0.088  &   0.073  &   0.070 \\ \hline
    $K_2$ & 0.058  &   0.087  &   0.079 \\\hline
  \end{tabular}
\end{center}
\end{table}

\begin{figure}[htb]
\centering
\includegraphics[width=12cm,height=7cm]{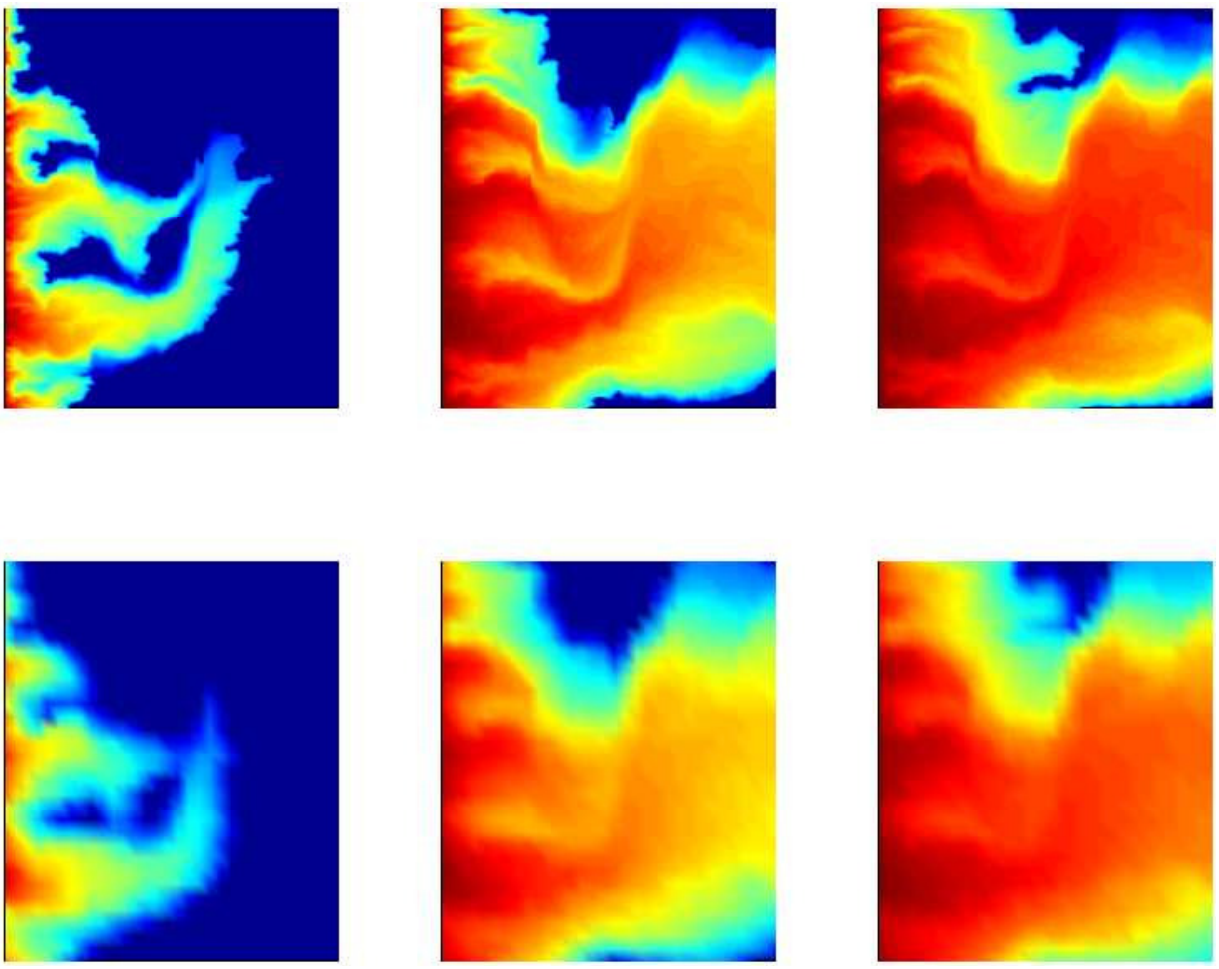}
\caption{The saturation profile using PG DG-LOD for the pressure
  equation on the grid $\T_H$ (bottom) and the reference solution on
  the grid $\T_h$ (upper) at time $T_1=0.05$ (left), $T_2=0.25$ (middle), and
  $T_3=0.45$ (right) using permeability $K_1$.}
\label{fig:saturation21}
\end{figure}

\begin{figure}[htb]
\centering
\includegraphics[width=12cm,height=7cm]{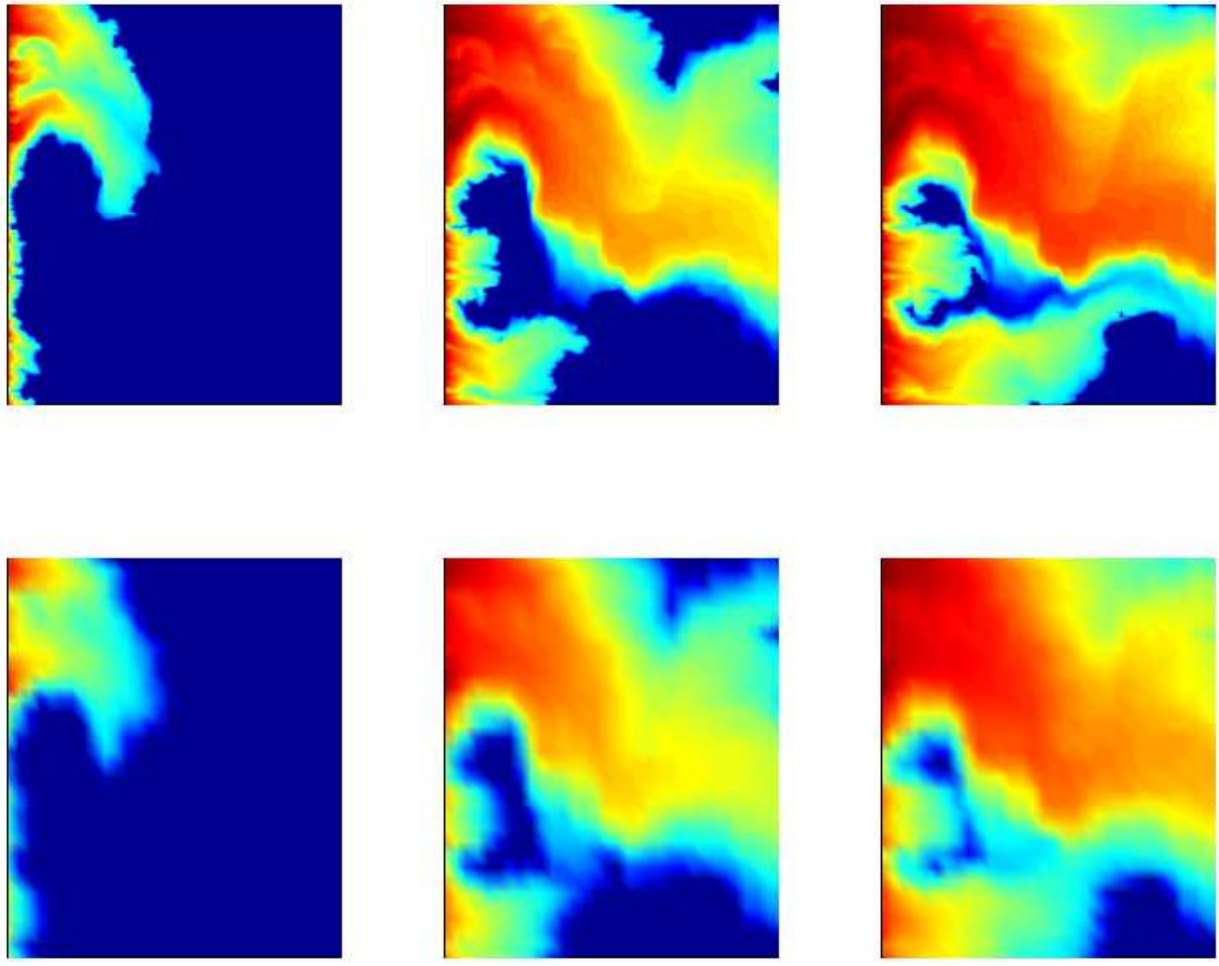}
\caption{The saturation profile using PG DG-LOD for the pressure
  equation on the grid $\T_H$ (bottom) and the reference solution on
  the grid $\T_h$ (upper) at time $T_1=0.05$ (left), $T_2=0.25$ (middle), and
  $T_3=0.45$ (right) using permeability $K_2$.}
\label{fig:saturation31}
\end{figure}

\begin{table}[htb]
  \caption{We consider the test case involving $K_1$. The table depicts relative $L^2$-errors between $S$ and $S^{\mathrm{ref}}$ at $T_2=0.25$ for different values of the coarse mesh size $H$. Here, $S^{\mathrm{ref}}$ denotes the reference solution computed on $T_h$ with $h=2^{-8}$ and $S$ denotes the numerical approximation obtained with the IMPES scheme, using the PG DG-LOD for solving the pressure equation (with coarse mesh $\T_H$). We pick $k= \lceil 2|\log(H)| \rceil$.}
  \label{table-BL-convergence}
  \centering
  \begin{tabular}{|c|c|}\hline
    $H$     & $\|e(T_2)\|_{L^2(\Omega)}$ \\ \hline\hline
    $2^{-3}$ & 0.220  \\ \hline
    $2^{-4}$ & 0.113  \\\hline
    $2^{-5}$ & 0.073  \\ \hline
    $2^{-6}$ & 0.048  \\\hline
  \end{tabular}
\end{table}

\section{Proofs of the main results}
\label{section:proofs:pg:lod}
\setcounter{equation}{0}
In this proof section we will frequently exploit the estimate
\begin{align}
\label{L2-norm-smaller-than-energy} \| v_h \|_{L^2(\Omega)} \lesssim |||v_h |||_h \qquad \mbox{for all } v_h \in V_h,
\end{align}
which is a conclusion from assumption (A7). Let $I_H^{-1}:=(I_H\vert_{V_H})^{-1}$, then (\ref{L2-norm-smaller-than-energy}) can be verified as follows by using (A7).
\begin{align*}
\| v_h \|_{L^2(\Omega)} &\le \| v_h - I_H(v_h) \|_{L^2(\Omega)} + \| I_H ( v_h ) \|_{L^2(\Omega)}\\
&\lesssim  H ||| v_h |||_h + \| (I_H \circ I_H^{-1} \circ I_H) ( v_h ) \|_{L^2(\Omega)} \\
&\lesssim  H ||| v_h |||_h + ||| (I_H^{-1} \circ I_H) ( v_h ) |||_H \lesssim  H ||| v_h |||_h + |||  I_H ( v_h ) |||_H \lesssim  H ||| v_h |||_h + |||  v_h |||_h.
\end{align*}

\subsection{Proof of Theorem~\ref{t:a-priori-local}}
The arguments for establishing the error estimate in $|||\cdot|||_h$-norm is analogous to the standard case, see for example, \cite{Malqvist:Peterseim:2011} or \cite{Henning:Malqvist:2013}. We only recall the main arguments.

\begin{proof}[Proof of Theorem \ref{t:a-priori-local}]
Let $u_H^{\text{\tiny G-LOD}} = (u_H + Q_h(u_H)) \in \Vms$ be the Galerkin LOD solution governed by \eqref{lod-problem-eq}.
Utilizing the notation in (A8), we set $u_{H,\Omega}\in V_H$ to satisfy
\begin{align*}
a_h( u_{H,\Omega} + Q_h^{\Omega}(u_{H,\Omega}), \Phi_H + Q_h^{\Omega}(\Phi_H) ) = ( f ,\Phi_H + Q_h^{\Omega}(\Phi_H) ) \quad \mbox{for all } \Phi_H \in V_H
\end{align*}
and define $e_h:=u_{H,\Omega} + Q_h^{\Omega}(u_{H,\Omega}) - u_h$. Using Galerkin orthogonality, we obtain $a_h( e_h, \Phi ) = 0$ for all $\Phi \in \Vms_{\Omega}$ and hence $e_h \in W_h$ (i.e. $I_H(e_h)=0$). This implies $||| e_h |||_h^2 \lesssim a_h(e_h,e_h) = (f,e_h) = (f, e_h -I_H(e_h))
 \lesssim H \| f\|_{L^2(\Omega)} \hspace{2pt}||| e_h |||_h$ and consequently by energy minimization
\begin{eqnarray*}
\lefteqn{||| u_H^{\text{\tiny G-LOD}} - u_h |||_h = ||| u_H + Q_h(u_H) - u_h |||_h \lesssim ||| u_{H,\Omega} + Q_h(u_{H,\Omega}) - u_h |||_h}\\
&\le&  ||| e_h |||_h + ||| Q_h^{\Omega}(u_{H,\Omega}) - Q_h(u_{H,\Omega}) |||_h \overset{\mbox{(A8)}}{\lesssim}
H \| f\|_{L^2(\Omega)} + (1/H)^p k^{d/2} \theta^k  ||| u_{H,\Omega} + Q_h^{\Omega}(u_{H,\Omega})  |||_h.
\end{eqnarray*}
The bound $||| u_{H,\Omega} + Q_h^{\Omega}(u_{H,\Omega})  |||_h \lesssim \| f\|_{L^2(\Omega)}$ finishes the energy-norm estimate. The estimate in the $L^2$-norm is established in a similar fashion using (\ref{L2-norm-smaller-than-energy}).
\end{proof}

\subsection{Proof of Theorem~\ref{t:a-priori-local-pg}}
We begin with stating and proving a lemma that is required to establish the a priori error estimate.
\begin{lemma}
\label{lemma-1}
For all $v^{{\rm ms}} \in \Vms_{\Omega}$ with $v^{{\rm ms}}=v_H + v^f$, where $v_H\in V_H$ and $v^f \in W_h$, we have
\begin{align}
\label{eqn-lemma-1}\| v^f \|_{L^2(\Omega)} \lesssim H ||| v^{\rm ms} |||_h.
\end{align}
\end{lemma}

\begin{proof}
Because of $I_H(v^f)=0$ and $(I_H^{-1} \circ I_H)(v_H) =v_H$,
\begin{align*}
v^f = v^f - I_{H}(v^f) + v_H - (I_H^{-1} \circ I_H)(v_H+v^f) + I_H(v_H+v^f) - I_H(v_H),
\end{align*}
and therefore with $I_H= I_H \circ I_H^{-1} \circ I_H$ and (A7),
\begin{align*}
\| v^f\|_{L^2(\Omega)} &\le \| v^{\rm ms} - I_{H}(v^{\rm ms}) \|_{L^2(\Omega)} + \| (I_H^{-1} \circ I_H)(v^{\rm ms}) - I_H(v^{\rm ms}) \|_{L^2(\Omega)} \\
&\lesssim H ||| v^{\rm ms} |||_h + \| (I_H^{-1} \circ I_H)(v^{\rm ms}) - (I_H \circ I_H^{-1} \circ I_H)(v^{\rm ms}) \|_{L^2(\Omega)} \\
&\lesssim H ||| v^{\rm ms} |||_h + H ||| (I_H^{-1} \circ I_H)(v^{\rm ms}) |||_H \\
&\lesssim H ||| v^{\rm ms} |||_h.
\end{align*}
In the last step we used again the stability estimates for $I_H^{-1}$ and $I_H$ in (A7).
\end{proof}

\begin{proof}[Proof of Theorem \ref{t:a-priori-local-pg}]
Let $u_{H,\Omega}^{\text{\tiny G-LOD}}$ and $u_{H,\Omega}^{\text{\tiny PG-LOD}}$ be respectively the solution of (\ref{lod-problem-eq}) and
(\ref{lod-problem-eq-pg}) for $U(T)=\Omega$. As in the statement of the theorem, $u_{H}^{\text{\tiny PG-LOD}}$  is the solution of (\ref{lod-problem-eq-pg}) for $U(T)=U_k(T)$.  By adding and subtracting appropriate terms and applying triangle inequality, we arrive at
\begin{equation*}
||| u_h- u_{H}^{\text{\tiny PG-LOD}} |||_h \le \mbox{I}_1 + \mbox{I}_2 + \mbox{I}_3,
\end{equation*}
where $\mbox{I}_1 = ||| u_h- u_{H,\Omega}^{\text{\tiny G-LOD}} |||_h$, $\mbox{I}_2  = ||| u_{H,\Omega}^{\text{\tiny G-LOD}} - u_{H,\Omega}^{\text{\tiny PG-LOD}}|||_h$, and $\mbox{I}_3  = ||| u_{H,\Omega}^{\text{\tiny PG-LOD}} - u_{H}^{\text{\tiny PG-LOD}}|||_h$.
In the following, we estimate these three terms.
Because $e^{(1)} := (u_h-u_{H,\Omega}^{\text{\tiny G-LOD}}) \in W_h$  (c.f. proof of Theorem \ref{t:a-priori-local}) and by applying the Galerkin orthogonality, we get
\begin{equation}
\mbox{I}^2_1 \lesssim a_h(e^{(1)},e^{(1)}) = a_h(u_h,e^{(1)}) = (f,e^{(1)}-I_H(e^{(1)})) \lesssim H \|f \| \hspace{2pt} ||| e^{(1)} |||_h \le 
H \|f \| \hspace{2pt} \mbox{I}_1,
\end{equation}
i.e. $\mbox{I}_1 \hspace{1pt} \lesssim H \|f \|$.
Furthermore, $e^{(2)}:=(u_{H,\Omega}^{\text{\tiny PG-LOD}} -u_{H,\Omega}^{\text{\tiny G-LOD}}) \in \Vms_{\Omega}$ and the splitting $e^{(2)}=e^{(2)}_H+e^{(2)}_f$ with $e^{(2)}_H \in V_H$ and $e^{(2)}_f\in W_h$ (i.e. $I_H(e^{(2)}_f)=0$) holds true. Because
$a_h(u_{H,\Omega}^{\text{\tiny PG-LOD}},e^{(2)}_f) = 0$, we obtain
\begin{equation}
\mbox{I}^2_2 \lesssim a_h(e^{(2)} , e^{(2)} ) =
a_h(u_{H,\Omega}^{\text{\tiny PG-LOD}}, e^{(2)}_H ) - a_h(u_{H,\Omega}^{\text{\tiny G-LOD}}, e^{(2)} ) =
( f, e^{(2)}_H - e^{(2)} ) = -(f,e^{(2)}_f).
\end{equation}
By Lemma~\ref{lemma-1}, we know that $(f,e^{(2)}_f) \le \| f \|_{L^2(\Omega)} \, \| e^{(2)}_f \|_{L^2(\Omega)} \lesssim \| f \|_{L^2(\Omega)} \, H ||| e^{(2)} |||_h = H \| f \|_{L^2(\Omega)} \, \mbox{I}_2$. Again, we conclude
that $\mbox{I}_2 \hspace{1pt} \lesssim  H \|f\|_{L^2(\Omega)}$.

It remains to estimate $\mbox{I}_3$ for which we define $e^{(3)}:= u_{H,\Omega}^{\text{\tiny PG-LOD}} -u_{H}^{\text{\tiny PG-LOD}}$. To simplify the notation, we subsequently denote (according to the definitions of $\Vms$ and $\Vms_{\Omega}$)
$$
u_{H}^{\text{\tiny PG-LOD}}=u_H + Q_h(u_H)  \quad \mbox{and} \quad
u_{H,\Omega}^{\text{\tiny PG-LOD}}=u_H^{\Omega} + Q_h^{\Omega}(u_H^{\Omega}),
$$
where $u_H \in V_H$ and $u_H^{\Omega} \in V_H$. By the definition of problem (\ref{lod-problem-eq-pg}) we have
\begin{align}
\label{help-eqn-in-proof-thrm-3-4-a}a_h(u_{H}^{\text{\tiny PG-LOD}}, \Phi_H) = (f,\Phi_H) = a_h(u_{H,\Omega}^{\text{\tiny PG-LOD}}, \Phi_H).
\end{align}
On the other hand, by the definition of $Q_h^{\Omega}=-P_h$ (see Remark \ref{remark-on-projection}) and since $Q_h(\Phi_H) \in W_h$ we get
\begin{align}
\label{help-eqn-in-proof-thrm-3-4-b}
a_h(u_{H,\Omega}^{\text{\tiny PG-LOD}}, Q_h(\Phi_H)) = 0.
\end{align}
Combining (\ref{help-eqn-in-proof-thrm-3-4-a}) and (\ref{help-eqn-in-proof-thrm-3-4-b}) we get the equality
$$
a_h(u_{H}^{\text{\tiny PG-LOD}}, \Phi_H+Q_h(\Phi_H)) =
a_h(u_{H}^{\text{\tiny PG-LOD}}, Q_h(\Phi_H)) + a_h(u_{H,\Omega}^{\text{\tiny PG-LOD}}, \Phi_H + Q_h(\Phi_H)).
$$
We use this equality cast $u_H$ as a unique solution of a self-adjoint variational equation expressed as
\begin{align*}
a_h( u_H + Q_h(u_H), \Phi_H + Q_h(\Phi_H) ) = F_{u_H,u_H^{\Omega}}(\Phi_H) \quad \mbox{for all} \enspace \Phi_H \in V_H,
\end{align*}
where $F_{u_H,u_H^{\Omega}}$ is a given fixed data function written as
$$
F_{u_H,u_H^{\Omega}}(\Phi_H) =
a_h( u_H + Q_h(u_H), Q_h(\Phi_H) ) + a_h( u_H^{\Omega} + Q_h^{\Omega}(u_H^{\Omega}), \Phi_H + Q_h(\Phi_H) ).
$$
%
%
%
Since this problem is self-adjoint, we get that $u_H$ is equally the minimizer in $V_H$ of the functional
\begin{align*}
J(\Phi_H) := &a_h( \Phi_H + Q_h(\Phi_H) - u_H^{\Omega} - Q_h^{\Omega}(u_H^{\Omega}), \Phi_H + Q_h(\Phi_H) - u_H^{\Omega} - Q_h^{\Omega}(u_H^{\Omega}) )\\
&- 2 a_h( u_H + Q_h(u_H), Q_h(\Phi_H) ).
\end{align*}
Hence we obtain
\begin{equation} \label{eq:zzz} 
\begin{aligned}
\alpha \mbox{I}_3^2 &= \alpha ||| e^{(3)} |||_h^2\\
&\le a_h( e^{(3)}, e^{(3)} )\\
&= J( u_H ) + 2 a_h( u_H + Q_h(u_H), Q_h(u_H) ) \\
&\le J( u_H^\Omega ) + 2 a_h( u_H + Q_h(u_H), Q_h(u_H) )\\
&=  a_h( Q_h(u_H^{\Omega}) - Q_h^{\Omega}(u_H^{\Omega}), Q_h(u_H^{\Omega}) - Q_h^{\Omega}(u_H^{\Omega}) )\\
&\qquad - 2 a_h( u_H + Q_h(u_H), Q_h(u_H) - Q_h(u_H^{\Omega}) )\\
&= \mbox{I}_{31} + \mbox{I}_{32},
\end{aligned}
\end{equation}
where
$$
\begin{aligned}
\mbox{I}_{31} &= a_h( Q_h(u_H^{\Omega}) - Q_h^{\Omega}(u_H^{\Omega}), Q_h(u_H^{\Omega}) - Q_h^{\Omega}(u_H^{\Omega}) )\\
\mbox{I}_{32} &= a_h( Q_h(u_H) -  Q_h^{\Omega}(u_H), Q_h(u_H) - Q_h(u_H^{\Omega}) ).
\end{aligned}
$$
%
%
By the boundedness of $a_h(\cdot,\cdot)$ and applying \eqref{equation-influence-intersections} we get
\begin{equation} \label{eq:fei}
\mbox{I}_{31} \lesssim ||| Q_h(u_H^{\Omega}) - Q_h^{\Omega}(u_H^{\Omega}) |||_h^2 \lesssim
k^{p} \theta^{2k} (1/H)^{2p} ||| u_H^{\Omega} + Q_h^{\Omega}(u_H^{\Omega}) |||^2_h.
\end{equation}
We now need to estimate $u_{H,\Omega}^{\text{\tiny PG-LOD}} = u_H^{\Omega} + Q_h^{\Omega}(u_H^{\Omega})$.
By the inf-sup condition and Lemma~\ref{lemma-1},
\begin{equation}
\label{energy-estimate-pg-lod}
\begin{aligned}
||| u_{H,\Omega}^{\text{\tiny PG-LOD}}  |||^2_h
&\lesssim a_h(u_{H,\Omega}^{\text{\tiny PG-LOD}} , u_{H,\Omega}^{\text{\tiny PG-LOD}})\\
&=  a(u_{H,\Omega}^{\text{\tiny PG-LOD}} , u_H^{\Omega} )\\
&= (f, u_H^{\Omega} )\\
&= (f, u_{H,\Omega}^{\text{\tiny PG-LOD}} ) - (f, Q_h^{\Omega}(u_H^{\Omega}) ) \\
&\lesssim (1+H) \| f \|_{L^2(\Omega)} \hspace{3pt} ||| u_{H,\Omega}^{\text{\tiny PG-LOD}} |||_h,
\end{aligned}
\end{equation}
and thus combining it with \eqref{eq:fei} yields
\begin{equation} \label{eq:ccc}
\mbox{I}_{31} \lesssim k^{d} \theta^{2k} (1/H)^{2p}  \| f \|^2_{L^2(\Omega)} 
\end{equation}
%
%
Furthermore, in a similar fashion we use the boundedness of $a_h(\cdot,\cdot)$ and \eqref{equation-influence-intersections} to get
\begin{equation} \label{eq:xxx}
\begin{aligned}
\mbox{I}_{32} &\lesssim ||| Q_h(u_H) -  Q_h^{\Omega}(u_H) |||_h ~ ||| Q_h(u_H) - Q_h(u_H^{\Omega})  |||_h \\
&\lesssim k^{d/2} \theta^{k} (1/H)^p ||| u_{H}^{\text{\tiny PG-LOD}} |||_h ~  ||| Q_h(u_H) - Q_h(u_H^{\Omega})  |||_h 
\end{aligned}
\end{equation}
By adding and subtracting appropriate terms and applying triangle inequality
\begin{equation} \label{eq:ttt}
||| Q_h(u_H) - Q_h(u_H^{\Omega})  |||_h \le
||| (Q_h - Q_h^\Omega)(u_H) |||_h + ||| Q_h^\Omega(u_H- u_H^\Omega) |||_h + ||| (Q_h^\Omega - Q_h)(u_H^{\Omega})  |||_h. 
\end{equation}
We use \eqref{equation-influence-intersections}  to estimate the first and last terms in \eqref{eq:ttt} to yield
\begin{equation} \label{eq:qiuqiu}
||| (Q_h - Q_h^\Omega)(u_H) |||_h + ||| (Q_h^\Omega - Q_h)(u_H^{\Omega})  |||_h \lesssim
k^{d/2} \theta^{k} (1/H)^p ( ||| u_{H}^{\text{\tiny PG-LOD}} |||_h + ||| u_{H,\Omega}^{\text{\tiny PG-LOD}} |||_h)
\end{equation}
Moreover, by the $|||\cdot|||_h$-stability of $Q_h^{\Omega}$ (which holds true since $Q_h^{\Omega}=-P_h$ with $P_h$ being the orthogonal projection defined in (\ref{projection-orthogonal})), we have
\begin{equation} \label{eq:qiuqiumid}
||| Q_h^\Omega(u_H- u_H^\Omega) |||_h \lesssim  ||| u_H - u_H^{\Omega} |||_h = ||| ((I_H\vert_{V_H})^{-1} \circ I_H)(e^{(3)}) |||_h 
\lesssim  C_{H,h} |||  e^{(3)} |||_h.
\end{equation}
Putting back \eqref{eq:qiuqiumid} and \eqref{eq:qiuqiu} to \eqref{eq:ttt} and place it in \eqref{eq:xxx} gives
\begin{equation}
\begin{aligned}
\mbox{I}_{32} &\lesssim k^{d} \theta^{2k} (1/H)^{2p} ||| u_{H}^{\text{\tiny PG-LOD}} |||_h ( ||| u_{H}^{\text{\tiny PG-LOD}} |||_h + ||| u_{H,\Omega}^{\text{\tiny PG-LOD}} |||_h ) \\
&\hspace*{0.5cm}+  k^{d/2} \theta^{k} (1/H)^p ||| u_{H}^{\text{\tiny PG-LOD}} |||_h ~C_{H,h} |||  e^{(3)} |||_h\\
&\lesssim k^{d} \theta^{2k} (1/H)^{2p} ( ||| u_{H}^{\text{\tiny PG-LOD}} |||^2_h + ||| u_{H,\Omega}^{\text{\tiny PG-LOD}} |||^2_h )\\
&\hspace*{0.5cm}+ \frac{C^2_{H,h}}{\delta} k^{d} \theta^{2k} (1/H)^{2p} ||| u_{H}^{\text{\tiny PG-LOD}} |||^2_h +  \frac{\delta}{4} |||  e^{(3)} |||^2_h,
\end{aligned}
\end{equation}
where in the last step we use the Young's inequality for both terms, and in particular for the second term, inserting a sufficiently small $\delta>0$ so that we can later on hide the term $\frac{\delta}{4} ||| e^{(3)} |||_h^2$ in the left hand side of (\ref{eq:zzz}). Note that the choice of $\delta$ is independent of $H$, $h$ or $k$. Rearranging and collecting common terms in the last inequality gives
$$
\mbox{I}_{32} \lesssim
k^d \theta^{2k} (1/H)^{2p} \left( (1+\frac{C^2_{H,h}}{\delta}) ||| u_{H}^{\text{\tiny PG-LOD}} |||^2 + ||| u_{H,\Omega}^{\text{\tiny PG-LOD}} |||^2 \right) + \frac{\delta}{4} ||| e^{(3)} |||^2_h,
$$
so that we need to estimate $|||u_{H}^{\text{\tiny PG-LOD}}|||_h$ and $|||u_{H,\Omega}^{\text{\tiny PG-LOD}}|||_h$, respectively. The stability of the second piece was established in (\ref{energy-estimate-pg-lod}), while the stability of the first piece is achieved by employing (A9) and (A7) in
$$
\bar{\alpha} ||| u_{H}^{\text{\tiny PG-LOD}} |||_h \hspace{2pt} ||| u_H |||_H \lesssim a_h( u_{H}^{\text{\tiny PG-LOD}}, u_H ) = (f,u_H) \lesssim \|f\|_{L^2(\Omega)} \hspace{3pt} ||| u_H |||_H,
$$
from which we conclude that
\begin{align*}
\mbox{I}_{32} \lesssim k^d \theta^{2k} (1/H)^{2p} \left( (1+\frac{C^2_{H,h}}{\delta}) (1+ \bar{\alpha}^{-1}) \| f \|^2 \right) + \frac{\delta}{4}  \mbox{I}_3^2.
\end{align*}
To summarize, putting this last inequality and \eqref{eq:ccc} to \eqref{eq:zzz} and choosing sufficiently small $\delta$ gives
$$
\mbox{I}_3 \lesssim k^{d/2} \theta^{k} (1/H)^{p} \left( (1+\frac{C_{H,h}}{\delta}) (1+ \bar{\alpha}^{-1}) \| f \| \right),
$$
combining it with the existing estimates for $\mbox{I}_1$ and $\mbox{I}_2$ proves the error estimate in $||| \cdot |||_h$.
Moreover, the estimate in the $L^2$-norm is established in a similar fashion. This completes the proof of the theorem.
\end{proof}

\subsection{Proof of Lemma \ref{inf-sup-stability-pg-lod} and Lemma \ref{inf-sup-stability-dg-pg-lod}}
Next, we prove the inf-sup stability of the Continuous Galerkin LOD in Petrov-Galerkin formulation.

\begin{proof}[Proof of Lemma \ref{inf-sup-stability-pg-lod}]
Let $\Phi^{\operatorname*{ms}} \in \Vms$ be an arbitrary element. To prove the inf-sup condition, we aim to show that
\begin{align}
\label{to-show-proof-inf-sup-stability} \frac{a_h(\Phi^{\operatorname*{ms}},\Phi_H)}{||| \Phi_H |||_h}
 \ge \alpha(k) |||  \Phi^{\operatorname*{ms}} |||_h \hspace*{0.3cm} \text{for} \hspace*{0.3cm} 
\Phi_H = ((I_H\vert_{V_H})^{-1} \circ I_H)( \Phi^{\operatorname*{ms}} ). 
\end{align}
Let therefore $U(T)=U_k(T)$ for fixed $k\in \mathbb{N}$. By the definitions of $\Vms$ and $\Phi_H$, we have $\Phi^{\operatorname*{ms}} = \Phi_H + Q_h(\Phi_H)$, where $Q_h(\Phi_H)$ denotes the corresponding corrector given by (\ref{global-corrector}). By $Q_h^{\Omega}(\Phi_H)$ we denote the corresponding global corrector for the case $U(T)=\Omega$ and the local correctors are denoted by $Q_h^{\Omega,T}(\Phi_H)$.
First, we observe that by $||| \cdot |||_h=||| \cdot |||_H$
\begin{align}
\label{coarse-against-total-estimate}||| \Phi_H |||_h &= ||| ((I_H\vert_{V_H})^{-1} \circ I_H)(\Phi^{\operatorname*{ms}}) |||_h \lesssim ||| \Phi^{\operatorname*{ms}} |||_h,
\end{align}
where we used the $|||\cdot|||_h$-stability of $I_H$ and $(I_H\vert_{V_H})^{-1}$ according to (A7). Consequently,  (\ref{coarse-against-total-estimate})
implies
\begin{align}
\label{fine-against-total-estimate}||| Q_h(\Phi_H) |||_h &\le 
||| \Phi^{\operatorname*{ms}} |||_h + ||| \Phi_H |||_h \lesssim |||\Phi^{\operatorname*{ms}} |||_h,
\end{align}
and thus
\begin{equation}
\label{inf-sup-est}
\begin{aligned}
a_h( \Phi^{\operatorname*{ms}} , \Phi_H ) 
&= a_h( \Phi^{\operatorname*{ms}}  , \Phi^{\operatorname*{ms}} ) - a_h( \Phi^{\operatorname*{ms}} , Q_h(\Phi_H)) \\
&\ge \alpha ||| \Phi^{\operatorname*{ms}} |||_h^2 - a_h( \Phi^{\operatorname*{ms}} , Q_h(\Phi_H)) \\
&\ge C \alpha ||| \Phi_H |||_h \hspace{2pt} ||| \Phi^{\operatorname*{ms}} |||_h - a_h( \Phi^{\operatorname*{ms}} , Q_h(\Phi_H)),
\end{aligned}
\end{equation}
where we have used \eqref{coarse-against-total-estimate} again to bound $||| \Phi^{\operatorname*{ms}} |||_h$ from below. Note here that
$C$ denotes a generic constant. It remains to bound $a_h( \Phi^{\operatorname*{ms}} , Q_h(\Phi_H))$.
By the orthogonality of $\Vms_{\Omega}$ and $W_h$ we have
\begin{align}
\label{lemma-orthogonality-opt-corr} a_h( \Phi_H + Q_h^{\Omega}(\Phi_H) , Q_h(\Phi_H) ) = 0,
\end{align}
and since $a_h(\cdot,\cdot)$ is such that $a_h(v_h,w_h)=0$ for all $v_h,w_h \in V_h$ with supp$(v_h)\cap$supp$(w_h)=\emptyset$ we get by the definition of $Q_h(\Phi_H)$ for every $w_h^T\in W_h(T)$
\begin{align}
\label{lemma-orthogonality-loc-corr} \nonumber a_h( \Phi_H + Q_h(\Phi_H) , w_h^T ) &=
\sum_{K \in \mathcal{T}_H} \left( a_h^K(  \Phi_H , w_h^T ) + a_h( Q_h(\Phi_H) , w_h^T ) \right)\\
\nonumber &= \left( \sum_{K \in \mathcal{T}_H} a_h^K(  \Phi_H , w_h^T ) \right) + a_h( Q_h^T(\Phi_H) , w_h^T )\\
&= a_h( \Phi_H + Q_h^T(\Phi_H) , w_h^T ) = 0.
\end{align}
Using both equalities above and by the boundedness of $a_h(\cdot,\cdot)$ and applying \eqref{fine-against-total-estimate} yields
\begin{equation}
\begin{aligned}
a_h( \Phi^{\operatorname*{ms}} , Q_h(\Phi_H))
&= a_h( \Phi_H + Q_h^{\Omega}(\Phi_H) , Q_h(\Phi_H)) + a_h(Q_h(\Phi_H) - Q_h^{\Omega}(\Phi_H) , Q_h(\Phi_H))\\
&= a_h(Q_h(\Phi_H) - Q_h^{\Omega}(\Phi_H) , Q_h(\Phi_H) - w_h )\\
&\le ||| Q_h(\Phi_H) - Q_h^{\Omega}(\Phi_H) |||_h \hspace{2pt} \frac{||| Q_h(\Phi_H) - w_h |||_h}{||| Q_h(\Phi_H) |||_h} ||| \Phi^{\operatorname*{ms}} |||_h\\
\end{aligned}
\end{equation}
We next estimate $||| Q_h(\Phi_H) - Q_h^{\Omega}(\Phi_H) |||_h$ by applying \eqref{equation-influence-intersections-cg} and establishing an analog
of (\ref{local-energy}) for $Q_h^{\Omega,T}(\Phi_H)$ expressed as
\begin{align}
\label{local-energy-estimate-proof} ||| Q_h^{\Omega,T}(\Phi_H) |||_h^2 & \lesssim ||| \Phi_H |||_{h,T} ||| Q_h^{\Omega,T}(\Phi_H) |||_h,
\end{align}
giving (for $k>0$)
\begin{equation}
\begin{aligned}
 ||| Q_h(\Phi_H) - Q_h^{\Omega}(\Phi_H) |||_h &\lesssim k^{d/2} \theta^{k} \left( \sum_{T\in\T_H} ||| Q_h^{\Omega,T}(\Phi_H) |||_h^2 \right)^{1/2} \\
&\lesssim k^{d/2} \theta^{k} \left( \sum_{T\in\T_H} ||| \Phi_H |||_{h,T}^2 \right)^{1/2}\\
&\lesssim k^{d/2} \theta^{k}  ||| \Phi_H |||_h.
 \end{aligned}
\end{equation}
Thus we end up with
\begin{equation}
a_h( \Phi^{\operatorname*{ms}} , Q_h(\Phi_H)) \lesssim
\Bigg(  \frac{||| Q_h(\Phi_H) - w_h |||_h}{||| Q_h(\Phi_H) |||_h} \Bigg) k^{d/2} \theta^k ||| \Phi_H |||_h ~  ||| \Phi^{\operatorname*{ms}} |||_h,
\end{equation}
which when combined with
(\ref{inf-sup-est}) implies that there exist positive generic constants $C_1,C_2$ (independent of $H$ and $k$) such that
\begin{eqnarray}
\label{final-estimate}\frac{a_h( \Phi^{\operatorname*{ms}} , \Phi_H ) }{||| \Phi_H |||_h \hspace{2pt} ||| \Phi^{\operatorname*{ms}} |||_h}
&\ge& C_1 \alpha - C_2 k^{d/2} \theta^{k} \inf_{w_h \in W_h^T} \frac{ ||| Q_h(\Phi_H) -  w_h |||_h}{||| Q_h(\Phi_H) |||_h}.
\end{eqnarray}
Since $\inf_{w_h \in W_h^T} \frac{ ||| Q_h(\Phi_H) -  w_h |||_h}{||| Q_h(\Phi_H) |||_h}=0$ for $k=0$, estimate (\ref{final-estimate}) holds for all $k \in \mathbb{N}$ and the condition $k>0$ is not required.
The relation
$Q_h(\Phi_H) = \Phi^{\operatorname*{ms}} - ((I_H\vert_{V_H})^{-1} \circ I_H)( \Phi^{\operatorname*{ms}} )$
finishes the proof.
\end{proof}

Finally, we prove the inf-sup stability of the Discontinuous Galerkin LOD in Petrov-Galerkin formulation.

\begin{proof}[Proof of Lemma \ref{inf-sup-stability-dg-pg-lod}]
The main arguments are similar as in the proof of Lemma \ref{inf-sup-stability-pg-lod}. Set $n:=(m+3)/2$. Let $\Phi^{\operatorname*{ms}} = \Phi_H + Q_h(\Phi_H)\in \Vms$ be an arbitrary element and let $U(T)=U_k(T)$ for fixed $k \gtrsim n |\log(H)|$. By the assumptions on $\T_H$ and $\T_h$ and by the definitions of $||| \cdot |||_h$ and $||| \cdot |||_H$ it is easy to see that 
\begin{align*}
||| \Phi_H |||_h \lesssim H^{(1-m)/2} ||| \Phi_H |||_H \qquad \mbox{and} \qquad ||| \Phi_H |||_H \lesssim ||| \Phi^{\operatorname*{ms}} |||_h.
\end{align*}
Consequently we get
\begin{align}
\label{coarse-against-total-estimate-dg}
||| Q_h(\Phi_H) |||_h &\le 
||| \Phi^{\operatorname*{ms}} |||_h + ||| \Phi_H |||_h 
\lesssim (1+H^{(1-m)/2}) ||| \Phi^{\operatorname*{ms}} |||_h.
\end{align}
Thus
\begin{equation}
\label{inf-sup-est-dg}
\begin{aligned}
a_h( \Phi^{\operatorname*{ms}} , \Phi_H ) 
&= a_h( \Phi^{\operatorname*{ms}}  , \Phi^{\operatorname*{ms}} ) - a_h( \Phi^{\operatorname*{ms}} , Q_h(\Phi_H)) \\
&\ge \alpha ||| \Phi^{\operatorname*{ms}} |||_h^2  - a_h( \Phi^{\operatorname*{ms}} , Q_h(\Phi_H))\\
&= \alpha ||| \Phi^{\operatorname*{ms}} |||_h^2  - a_h(Q_h(\Phi_H) - Q_h^{\Omega}(\Phi_H) , Q_h(\Phi_H) ) \\
&\ge \alpha ||| \Phi^{\operatorname*{ms}} |||_h^2 - ||| Q_h(\Phi_H) - Q_h^{\Omega}(\Phi_H) |||_h \hspace{2pt} ||| Q_h(\Phi_H) |||_h\\
&\overset{(\ref{coarse-against-total-estimate-dg})}{\ge} \alpha ||| \Phi^{\operatorname*{ms}} |||_h^2 - ||| Q_h(\Phi_H) - Q_h^{\Omega}(\Phi_H) |||_h \hspace{2pt} (1+H^{(1-m)/2}) ||| \Phi^{\operatorname*{ms}} |||_h.
\end{aligned}
\end{equation}
Using
\begin{align*}
||| Q_h(\Phi_H) - Q_h^{\Omega}(\Phi_H) |||_h &\le C (1/H) k^{d/2} \theta^{k} ||| \Phi_H + Q_h^{\Omega}(\Phi_H) |||_h \\
&\le C (1/H) k^{d/2} \theta^{k} \left( ||| \Phi^{\operatorname*{ms}} |||_h + ||| Q_h(\Phi_H) - Q_h^{\Omega}(\Phi_H) |||_h \right)\\
&\le C  H^{n-1} \left( ||| \Phi^{\operatorname*{ms}} |||_h + ||| Q_h(\Phi_H) - Q_h^{\Omega}(\Phi_H) |||_h \right)
\end{align*}
we obtain that we have for small enough $H$
\begin{align*}
||| Q_h(\Phi_H) - Q_h^{\Omega}(\Phi_H) |||_h &\lesssim H^{n-1} ||| \Phi^{\operatorname*{ms}} |||_h.
\end{align*}
Inserting this into (\ref{inf-sup-est-dg}) gives us
\begin{align*}
a_h( \Phi^{\operatorname*{ms}} , \Phi_H )  & \ge \alpha ||| \Phi^{\operatorname*{ms}} |||_h^2 -
(1+H^{(1-m)/2}) H^{n-1} ||| \Phi^{\operatorname*{ms}} |||_h^2 
\ge C_1 (\alpha - C_2 H) ||| \Phi^{\operatorname*{ms}} |||_h^2. 
\end{align*}
If $H$ is small enough so that $(\alpha - C_2 H)$ is positive, the estimate $||| \Phi_H |||_H \lesssim ||| \Phi^{\operatorname*{ms}} |||_h$ concludes the inf-sup estimate.
\end{proof}

$\\$
{\bf{Acknowledgements.}} We would like to thank the anonymous referees for their valuable comments and their constructive feedback on the original manuscript which helped us to improve this article.

\end{document}